\documentclass{article}

\pdfoutput=1

\usepackage{amsmath,amsthm}
\newtheorem{lemma}{Lemma}
\newtheorem{theorem}[lemma]{Theorem}
\newtheorem{definition}[lemma]{Definition}
\newtheorem{corollary}[lemma]{Corollary}
\newtheorem{proposition}[lemma]{Proposition}
\theoremstyle{plain}
\newtheorem{example}[lemma]{Example}
\newtheorem{question}[lemma]{Question}

\usepackage{amssymb}
\usepackage{collcell}
\usepackage{comment}
\usepackage{array}
\newcolumntype{C}{>{$}c<{$}}
\newcolumntype{L}{>{$}l<{$}}
\newcolumntype{R}{>{$}r<{$}}
\newcolumntype{B}{>{\collectcell\foobox}c<{\endcollectcell}}
\def\foobox#1{\hbox to 5mm{\hfil$#1$\hfil}}
\usepackage{xspace}

\usepackage{tikz}
\usetikzlibrary{decorations.pathreplacing,positioning}

\usepackage{xcolor}
\colorlet{cB}{blue!70!black}
\colorlet{cR}{red!70!black}

\let\phi=\varphi

\newcommand{\Z}{\mathbb{Z}}
\newcommand{\N}{\mathbb{N}}
\newcommand{\NN}{\N_0}  % Non-Negative integers

\newcommand{\ST}{S} % as \Sigma     % set of STates
\newcommand{\STT}{Q} % was \Delta   % another set of STaTes
\newcommand{\ID}{\mathrm{id}}

\newcommand{\card}[1]{|#1|}
\newcommand{\eg}{e.g.\xspace}
\newcommand{\from}{:}
\newcommand{\itvl}[2]{[#1,#2)}
\newcommand{\itoinfty}[1]{\itvl{#1}{\infty}}
\newcommand{\minftytoi}[1]{(-\infty,#1)}
\newcommand{\mealy}{\mu}
\newcommand{\stairs}{\Psi}
\newcommand{\toleft}[1]{{#1}^{i-}}
\newcommand{\toright}[1]{{#1}^{i+}}
\newcommand{\arule}{slider}
\newcommand{\brule}{sweeper}

\newcommand{\SL}{\mathcal{L}}%\mathcal{S}^{LR}}
\newcommand{\SR}{\mathcal{R}}%\mathcal{S}^{RL}}

\newcommand{\CL}{\mathcal{L}^{cl}}
\newcommand{\CR}{\mathcal{R}^{cl}}

\usepackage{xcolor}
\colorlet{darkred}{red!70!black}

\newcommand{\loc}{f_\mathrm{loc}}

\newcommand{\ns}[2]{\left[ \begin{smallmatrix} #1 \\ #2 \end{smallmatrix} \right]}

\newcommand{\inst}[1]{\mbox{\textsuperscript{#1}}}

\title{Sequentializing cellular automata}

\author{Jarkko Kari\inst{1}\thanks{Research supported by the Academy of Finland grant 296018.}, \quad
  Ville Salo\inst{1}, \quad and 
  Thomas Worsch\inst{2}
  \\[5mm]
  \inst{1} University of Turku, Turku, Finland \\
  \inst{2} Karlsruhe Institute of Technology, Karlsruhe, Germany
}

\begin{document}

\maketitle

%======================================================================
\begin{abstract}
  We study the problem of sequentializing a cellular automaton without
  introducing any intermediate states, and only performing reversible
  permutations on the tape. We give a decidable characterization of
  cellular automata which can be written as a single left-to-right
  sweep of a bijective rule from left to right over an infinite tape.
\end{abstract}

\section{Introduction}

Cellular automata (CA) are models of parallel computation, so when
implementing them on a sequential architecture, one
cannot simply update the cells one by one -- some cells would see
already updated states and the resulting configuration would be
incorrect. The simplest-to-implement solution is to hold two copies of
the current configuration in memory, and map
$(x,x) \mapsto (x,G(x)) \mapsto (G(x),G(x))$. This is wasteful in
terms of memory, and one can, with a bit of thinking, reduce the
memory usage to a constant by simply remembering a `wave' containing
the previous values of the $r$ cells to the left of the current cell,
where $r$ is the radius of the CA.

Here, we study the situation where the additional memory usage can be,
in a sense, dropped to zero -- more precisely we remember \emph{only}
the current configuration $x \in S^\Z$, and to apply the cellular automaton we
sweep a permutation $\chi : S^m \to S^m$ from left to right over $x$
(applying it consecutively to all length-$m$ subwords of $x$). The
positions where the sweep starts may get incorrect values, but after a
bounded number of steps, the rule should start writing the image of
the cellular automaton. We formalize this in two ways, with `sliders'
and `sweepers', which are two ways of formally dealing with the
problem that sweeps `start from infinity'.

It turns out that the cellular automata that admit a sliding rule are
precisely the ones that are left-closing
(Definition~\ref{def:LeftClosing}), and whose number of right stairs
(see Definition~\ref{def:stair}) of length $3m$ divides $|\ST|^{3m}$
for large enough $m$. This can be interpreted as saying that the
average movement `with respect to any prime number' is not to the
right. See Theorem~\ref{thm:SliderCharacterization2} and
Theorem~\ref{thm:SweeperCharacterization} for the precise statements,
and Section~\ref{sec:Decidability} for decidability results.

We introduce the sweeping hierarchy where left-to-right sweeps and
right-to-left sweeps alternate, and the closing hierarchy where
left-closing and right-closing CA alternate. We show that the two
hierarchies coincide starting from the second step. We do not know if
the hierarchies collapse on a finite level.

%======================================================================
\subsection{Preliminaries}

We denote the set of integers by $\Z$.
For integers $i\leq j$ we write $[i,j)$ for
$\{ x\in\Z\mid i\leq x < j\}$ and $[i,j]$ for $[i,j)\cup\{j\}$;
furthermore $\itoinfty{i}=\{ x\in\Z \mid i\leq x\}$ and
$\minftytoi{i}=\{ x\in\Z \mid x< i\}$ have the obvious meaning.
Thus $\itoinfty{0}$ is the set of non-negative integers which is also
denoted by $\NN$.

Occasionally we use notation for a set $M$ of integers in a place
where a \emph{list} of integers is required.
If no order is specified we assume the natural increasing order.
If the reversed order is required we will write $M^R$.

For sets $A$ and $B$ the set of all functions $f\from A\to B$ is
denoted $B^A$.
For $f\in B^A$ and $M\subseteq A$ the restriction of $f$ to $M$ is
written as $f|_M$ or sometimes even $f_M$. Finite words $w\in\ST^n$
are lists of symbols, \eg mappings $w\from [0,n) \to \ST$. Number $n$
is the length of the word. The set of all finite words is denoted by
$\ST^*$.

Configurations of one-dimensional CA are biinfinite words
$x\from\Z\to\ST$.
Instead of $x(i)$ we often write $x_i$. We define the \emph{left
  shift} $\sigma \from S^\mathbb{Z} \to S^\mathbb{Z}$ by
$\sigma(x)_i = x_{i+1}$.
The restriction of $x$ to a subset $\minftytoi{i}$ gives a
left-infinite word for which we write $x_{\minftytoi{i}}$; for a
right-infinite word we write $x_{\itoinfty{i}}$. These are called
\emph{half-infinite words}.
Half-infinite words can also be shifted by $\sigma$, and this is
defined using the same formula. The domain is shifted accordingly so
for $x\in\ST^{[i,\infty)}$ we have $\sigma(x)\in\ST^{[i-1,\infty)}$.

We use a special convention for concatenating words: Finite words
`float', in the sense that they live in $\ST^n$ for some $n$, without
a fixed position, and $u \cdot v$ denotes the concatenation of $u$ and
$v$ as an element of $\ST^{|u| + |v|}$. Half-infinite configurations
have a fixed domain $(-\infty,i]$ or $[i,\infty)$ for some $i$, which
does not change when they are concatenated with finite words or other
half-infinite configurations, while finite words are shifted suitably
so that they fill the gaps exactly (and whenever we concatenate, we
make sure this makes sense).

More precisely, for $w \in \ST^*$ and $y \in \ST^{(-\infty, i]}$, we
have $y \cdot w \in \ST^{(-\infty,i+|w|]}$ and for $w \in \ST^*$ and
$z \in \ST^{[i,\infty)}$ we have $w \cdot z \in \ST^{[i-|w|,\infty)}$
(defined in the obvious way). For a word $w \in \ST^*$ and
half-infinite words $y \in \ST^{\minftytoi{i}}$ and
$z \in \ST^{\itoinfty{i+n}}$ we write $y \cdot w \cdot z$ for the
obvious configuration in $\ST^\Z$, and this is defined if and only if
$|w| = n$. 

The set $\ST^\Z$ of configurations is assigned the usual product
topology generated by cylinders. A \emph{cylinder} defined by word
$w\in\ST^n$ at position $i\in\Z$ is the set
\[
[w]_{[i,i+n)} = \{x\in\ST^\Z\ |\ x_{[i,i+n)}=w\}
\]
of configurations that contain word $w$ in position $[i,i+n)$.
Cylinders are open and closed, and the open sets in $\ST^\Z$ are
precisely the unions of cylinders. We extend the notation also to
half-infinite configurations, and define
$$
[y]_{D} = \{x\in\ST^\Z\ |\ x_{D}=y\}
$$
for $D = [i,\infty)$ and $D = (\infty, i]$, and any $y \in \ST^D$.
These sets are closed in the topology.

%======================================================================
\section{Sliders and sweepers}

A \emph{block rule} is a function $\chi : \ST^m \to \ST^m$. Given a
block rule $\chi$ we want to define what it means to ``apply $\chi$
from left to right once at every position''.
We provide two alternatives, compare them and characterize which
cellular automata can be obtained by them.
The first alternative, called a \arule, assumes a bijective block rule
$\chi$ that one can slide along a configuration left-to-right or
right-to-left to transition between a configuration $y$ and its image
$f(y)$.
The second alternative, called a \brule, must consistently provide
values of the image $f(y)$ when sweeping left-to-right across $y$
starting sufficiently far on the left.

We first define what it means to apply a block rule on a configuration.

\begin{definition}
  Let $\chi : \ST^m \to \ST^m$ be a block rule and $i\in \Z$. The
  application of $\chi$ at coordinate $i$ is the function
  $\chi^i:\ST^\Z\longrightarrow\ST^\Z$ given by
  $\chi^i(x)_{[i,i+m)} = \chi(x_{[i,i+m)})$ and $\chi^i(x)_j=x_j$ for
  all $j\not\in[i,i+n)$. More generally, for $i_1,\ldots,i_k \in \Z$
  we write
  \[
    \chi^{i_1,\ldots,i_k} = \chi^{i_k} \circ \cdots \circ \chi^{i_2}
    \circ \chi^{i_1}.
  \]
\end{definition}
When $m > 1$, it is meaningless to speak about ``applying $\chi$ to each
cell simultaneously'':
An application of $\chi$ changes the states of several cells at once.
Applying it slightly shifted could change a certain cell again, but in a
different way.

We next define finite and infinite sweeps of block rule applications 
with a fixed start position.
\begin{definition}
  Given a block rule $\chi$ for $i,j\in \Z$, $i\leq j$, define
  $\chi^{[i,j]}=\chi^j\circ\cdots\circ \chi^i$; analogously let
  $\chi^{[i,j)}=\chi^{j-1}\circ\cdots\circ \chi^i$.
  For any configuration $x\in\ST^{\Z}$ and fixed $i\in\Z$ the sequence
  of configurations $x^{(j)}=\chi^{[i,j]}(x)$ for $j\in\itoinfty{i}$
  has a limit (point in the topological space $\ST^{\Z}$) for which we
  write $\toright{\chi}(x)$.

  Analogously, for a block rule $\xi$ the sequence of configurations
  $x^{(j)} = \xi^{[j,i)^R}(x)$ for $j\in \minftytoi{i}$ has a limit
  for which we write $\toleft{\xi}(x)$.
\end{definition}
It should be observed that in the definition of $\toright{\chi}(x)$
one has $i<j$ and the block rule is applied at successive positions
from left to right.
On the other hand $j<i$ is assumed in the definition of
$\toleft{\xi}(x)$ and since the ${}^{{}^R}$ in $\xi^{[j,i)^R}$ indicates
application of $\xi$ at the positions in the \emph{reverse} order,
i.e.\ $i-1, i-2, \dots, j$, the block rule is applied from right to left.

The reason the limits always exist in the definition is that the value
of $x^{(j)}_i$ changes at most $m$ times, on the steps where the sweep
passes over the cell $i$.

\begin{example}
  \label{ex:swap}
  Let $S=\{0,1\}$ and consider the block rule
  $\chi\from S^{[0,2)}\to S^{[0,2)}: (a,b)\mapsto (b,a)$.
  For consistency with the above definition denote by $\xi$ the
  inverse of $\chi$ (which in this case happens to be $\chi$ again).
  Let $s\in S$ and $y\in S^{\Z}$.
  We will look at the configuration $x\in S^{\Z}$ with
  \[
    x_i =
    \begin{cases}
      y_{i+1}, &\text{ if } i<0 \\
      s, &\text{ if } i=0       \\
      y_{i}, &\text{ if } i>0
    \end{cases}
  \]
  The application of $\chi$ successively at positions $0, 1, 2, \dots$
  always swaps state $s$ with its right neighbor.
  Since cell $j$ can only possibly change when $\chi^{j-1}$ or
  $\chi^j$ is applied, each cell enters a fixed state after a finite
  number of steps; see also the lower part of Figure~\ref{fig:ex-swap}
  starting at the row with configuration $x$.

  \begin{figure}[ht]
    \centering
    \small
    \begin{tabular}[b]{R@{ \ $\cdots$}|*{7}{B|}c@{$\cdots$\quad}C}
      \cline{2-8}
      \xi^{0-}(x) & y_{-3} & y_{-2} & y_{-1} & y_0 & y_1 & y_2 & y_3 & & y \\
      \cline{2-8}
      \multicolumn{1}{c}{} & \multicolumn{7}{C}{\vdots} & \\
      \cline{2-8}
      \xi^{[-3,0)^R}(x) & s & y_{-2} & y_{-1} & y_0 & y_1 & y_2 & y_3 & \\
      \cline{2-8}
      \xi^{[-2,0)^R}(x) & y_{-2} & s & y_{-1} & y_0 & y_1 & y_2 & y_3 & \\
      \cline{2-8}
      \xi^{[-1,0)^R}(x) & y_{-2} & y_{-1} & s & y_0 & y_1 & y_2 & y_3 & \\
      \cline{2-8}
      \multicolumn{9}{c}{ }& \\
      \cline{2-8}
      x\hphantom{)} & y_{-2} & y_{-1} & y_0 & s & y_1 & y_2 & y_3 & & x\\
      \cline{2-8}
      \multicolumn{9}{c}{ }& \\
      \cline{2-8}
      \chi^{[0,1)}(x) & y_{-2} & y_{-1} & y_0 & y_1 & s & y_2 & y_3  & \\
      \cline{2-8}
      \chi^{[0,2)}(x) & y_{-2} & y_{-1} & y_0 & y_1 & y_2 & s & y_3  & \\
      \cline{2-8}
      \chi^{[0,3)}(x) & y_{-2} & y_{-1} & y_0 & y_1 & y_2 & y_3 & s & \\
      \cline{2-8}
      \multicolumn{1}{c}{} & \multicolumn{7}{C}{\vdots} & \\
      \cline{2-8}
      \chi^{0+}(x) & y_{-2} & y_{-1} & y_0 & y_1 & y_2 & y_3 & y_4 & & z \\
      \cline{2-8}
    \end{tabular}

    \caption{A sequence of configurations with the center cell at
      position $0$. Starting with configuration $x$ in the middle when
      going downward the swapping rule $\chi$ is applied to blocks
      $[0,1]$, $[1,2]$, etc., and going from $x$ upward
      rule $\xi=\chi$ is applied to blocks $[-1,0]$, $[-2,-1]$ and so
      on.}
    \label{fig:ex-swap}
  \end{figure}
\end{example}

\begin{example}
  \label{ex:RightXOR}
  Let $S=\{0,1\}$ and consider the block rule
  $\chi\from S^{[0,2)}\to S^{[0,2)}: (a,b)\mapsto (a+b,b)$. Then sliding
  this rule over a configuration $x \in \{0,1\}^\Z$ produces the image
  of $x$ in the familiar exclusive-or cellular automaton with
  neighborhood $\{0,1\}$ (elementary CA 102). We will see in
  Example~\ref{ex:LeftXOR} that the exclusive-or CA with neighborhood
  $\{-1,0\}$ can not be defined this way.
\end{example}

\subsection{Definition of \arule s}

\begin{definition}
  \label{def:arelation}
  A bijective block rule $\chi$ with inverse $\xi$ defines a
  \emph{\arule\ relation} $F \subset \ST^\Z \times \ST^\Z$ by
  $(y, z) \in F$ iff for some $x \in \ST^\Z$ and some $i\in\Z$ we have
  $\xi^{i-}(x) = y$ and $\chi^{i+}(x) = z$.
  We call the pair $(x,i)$ a \emph{representation} of $(y, z)\in F$.
\end{definition}
Note that every $(x,i)\in \ST^\Z\times\Z$ is a representation of
exactly one pair, namely $(\xi^{i-}(x),\chi^{i+}(x))\in F$.

\begin{lemma}
  \label{lem:repr}
  Let $(x,i)$ be a representation of $(y, z)\in F$ under a bijective
  block rule $\chi$ of block length $n$. Then
  $x_{\minftytoi{i}} = z_{\minftytoi{i}}$ and
  $x_{\itoinfty{i+n}} = y_{\itoinfty{i+n}}$.
\end{lemma}

\begin{proof}
  Applying block rule $\chi$ at positions $j\geq i$ in $x$ never
  changes cells at positions $k<i$. %
  Therefore $x_k = (\toright{\chi}(x))_k = z_k$ proving the
  first part.
  The second part follows analogously.
\end{proof}

\begin{lemma}
  \label{lem:reprlemma}
  Let $(y,z)\in F$ be fixed. For all $i\in\Z$
  denote
  $$R_i=\{x\in\ST^\Z\ |\ (x,i) \mbox{ is a representation of } (y,
  z)\}.$$ For $i<j$ the function
  $\chi^{[i,j)}: R_i\longrightarrow R_j$ is a bijection, with inverse
  $\xi^{[i,j)^R}$. All $R_i$ have the same finite cardinality.
\end{lemma}

\begin{proof}
  The claim follows directly from the definition and the facts that
  \begin{equation}
    \label{eq:reprlemma}
    \chi^{j+}\circ \chi^{[i,j)}=\chi^{i+} \mbox{ and }  \xi^{j-}\circ \chi^{[i,j)} =\xi^{i-},
  \end{equation}
  and that $\chi^{[i,j)}$ and $\xi^{[i,j)^R}$ are inverses of each other.

  More precisely, if $x\in R_i$ then
  $z=\chi^{i+}(x)=\chi^{j+}(\chi^{[i,j)}(x))$ and
  $y=\xi^{i-}(x)=\xi^{j-}(\chi^{[i,j)}(x))$ so
  $\chi^{[i,j)}(x)\in R_j$. This proves that $\chi^{[i,j)}$ maps $R_i$
  into $R_j$. This map is injective. To prove surjectivity, we show
  that for any $x'\in R_j$ its pre-image $\xi^{[i,j)^R}(x')$ is in
  $R_i$. Composing the formulas in (\ref{eq:reprlemma}) with
  $\xi^{[i,j)^R}$ from the right gives
  $\chi^{j+}=\chi^{i+} \circ \xi^{[i,j)^R}$ and
  $\xi^{j-} =\xi^{i-}\circ \xi^{[i,j)^R}$, so as above we get
  $z=\chi^{j+}(x')=\chi^{i+}(\xi^{[i,j)^R}(x'))$ and
  $y=\xi^{j-}(x')=\xi^{i-}(\xi^{[i,j)^R}(x'))$, as required.

  The fact that the cardinalities are finite follows from
  Lemma~\ref{lem:repr}: there are at most $|\ST|^n$ choices of
  $x_{[i,i+n)}$ in $x\in R_i$.
\end{proof}

\begin{lemma}
  \label{lem:AruleClosed}
  A \arule\ relation $F \subset \ST^\Z \times \ST^\Z$ defined by a
  bijective block rule $\chi$ is closed and shift-invariant, and the
  projections $(y,z)\mapsto y$ and $(y,z)\mapsto z$ map $F$
  surjectively onto $\ST^\Z$.
\end{lemma}

\begin{proof}
  By Lemma~\ref{lem:reprlemma} every $(y,z)\in F$ has a representation
  $(x,0)$ at position $0$. Therefore, the relation $F$ is closed as
  the image of $\ST^\Z$ in the continuous map
  $x \mapsto (\xi^{0-}(x), \chi^{0+}(x))$.

  Clearly $(x,i)$ is a representation of $(y,z)$ if and only if
  $(\sigma(x),i-1)$ is a representation of $(\sigma(y),\sigma(z))$.
  Hence the relation $F$ is shift-invariant.

  The image of $F$ under the projection $(y,z)\mapsto z$ is dense. To
  see this, consider any finite word $w$ and a configuration $x$ with
  $x_{[-|w|,0)}=w$. The pair $(x,0)$ represents some $(y,z)\in F$, and
  because $z=\chi^{0+}(x)$ we have $z_{[-|w|,0)}=w$. The denseness
  follows now from shift invariance and the fact that $w$ was
  arbitrary. The image of $F$ under the projection is closed so the
  image is the whole $\ST^\Z$.

  The proof for the other projection is analogous. 
\end{proof}

\begin{corollary}
  \label{cor:arule}
  If $F \subset \ST^\Z \times \ST^\Z$ defined by a bijective block
  rule $\chi$ is a function (that is, if for all $y\in \ST^\Z$ there
  is at most one $z\in \ST^\Z$ such that $(y,z)\in F$) then this
  function $f:y\mapsto z$ is a surjective cellular automaton.
\end{corollary}

\begin{proof}
  Because the projections $(y,z)\mapsto y$ and $(y,z)\mapsto z$ are
  onto, the function $f$ is full and surjective. Because the relation
  $F \subset \ST^\Z \times \ST^\Z$ is closed, the function $f$ is
  continuous. As it is continuous and shift-invariant, it is a
  cellular automaton.
\end{proof}

\begin{definition}
  \label{def:arule}
  Let $\chi$ be a bijective block rule such that the \arule\ relation
  it defines is a function $f:\ST^\Z \longrightarrow \ST^\Z$. The
  surjective cellular automaton $f$ is called the \emph{\arule}\
  defined by $\chi$.
\end{definition}

Example~\ref{ex:swap} indicates that the \arule\ for the block rule
swapping two states is the left shift. By Corollary~\ref{cor:arule}
every \arule\ is a surjective CA. But not every surjective CA is a
slider. This will follow from an exact characterization of which
cellular automata are \arule s below.

\subsection{Characterization of \arule s}

We start by improving Corollary~\ref{cor:arule}, by showing that
\arule s are left-closing cellular automata.

\begin{definition}
  \label{def:LeftClosing}
  Two configurations $y$ and $y'$ are \emph{right-asymptotic} if there
  is an index $i\in\Z$ such that $y_{\itoinfty{i}}=y'_{\itoinfty{i}}$.
  They are called \emph{left-asymptotic} if there is an index $i\in\Z$
  such that $y_{\minftytoi{i}} = y'_{\minftytoi{i}}$.
  A CA $f$ is \emph{left-closing} if for any two different
  right-asymptotic configurations $y$ and $y'$ we have
  $f(y)\not= f(y')$.
  Right-closing CA are defined symmetrically using left-asymptotic
  configurations.
\end{definition}

\begin{lemma}
  \label{lem:a-rule->lc}
  A \arule\ is a left-closing cellular automaton.
\end{lemma}

\begin{proof}
  Let \arule\ $f$ be defined by a bijective block rule
  $\chi : \ST^m \to \ST^m$, so that $f$ is a surjective cellular
  automaton. Let $\xi$ be the inverse of $\chi$.

  Suppose $f$ is not left-closing, so that there exist two distinct
  right-asymptotic configurations $y$ and $y'$ such that
  $f(y) = f(y')$. We may suppose the rightmost difference in $y$ and
  $y'$ is at the origin. Let $r$ be a radius for the local rule of
  $f$, where we may suppose $r \geq m$, and let
  $y_{[-2r, 2r]} = w_0 v, y'_{[-2r, 2r]} = w_1 v$ where
  $|w_1| = |w_2| = 2r+1$. We can apply the local rule of $f$ to words,
  shrinking them by $r$ symbols on each side, and write
  $F : \ST^* \to \ST^*$ for this map. Since $y$ and $y'$ have the same
  $f$-image, we have $F(w_0 v) = F(w_1
  v)$. 

  Let $n$ be such that $2^n > |\ST|^m$ and for each $k \in \{0,1\}^n$,
  define the configuration
  \[ y_k = ...0000 w_{k_1} v w_{k_2} v \cdots v w_{k_n} v . 0000... \]
  where the right tail of $0$s begins at the origin. For each $y_k$,
  pick a point $x_k$ representing $(y_k, f(y_k))$ at the origin. By
  the pigeon hole principle, there exist $k \neq k'$ such that
  $(x_k)_{[0,m)} = (x_{k'})_{[0,m)}$. Let $j$ be the maximal
  coordinate where $k$ and $k'$
  differ.

  Now, the rightmost difference in $y_k$ and $y_{k'}$ is in coordinate
  $R = -2r-1 - (4r+1)(n-j)$ (the last coordinate of the word
  $w_{k_j}$). We have
  $f(y_k)_{[R-r, \infty)} = f(y_{k'})_{[R-r, \infty)}$ by the
  assumption that $j$ is the rightmost coordinate where $k$ and $k'$
  differ, %\todo{where $x_k$ and $x_{k'}$ differ ?!},
  and by $F(w_0 v) = F(w_1 v)$. Thus we also have
  $(x_k)_{[R-r, 0)} = (x_{k'})_{[R-r, 0)}$, since
  $\chi^{0+}(x_k) = f(y_k)$ and $\chi^{0+}(x_{k'}) = f(y_{k'})$ and
  these sweeps do not modify coordinates in $[R-r, 0)$. Recall that we
  have $(x_k)_{[0,m)} = (x_{k'})_{[0,m)}$ by the choice of $k$ and
  $k'$, so $(x_k)_{[R-r, m)}$ and $(x_{k'})_{[R-r, m)}$. 

  Now, we should have $\xi^{0-}(x_k) = y_k$ and
  $\xi^{0-}(x_{k'}) = y_{k'}$, in particular we should have
  $\xi^{0-}(x_k)_R \neq \xi^{0-}(x_{k'})_R$. But this is impossible:
  $\xi^{0-}(x_k)_R$ is completely determined by $(x_k)_{[R-m+1, m)}$
  and similarly $\xi^{0-}(x_{k'})_R$ is determined by
  $(x_{k'})_{[R-m+1, m)}$, but
  $(x_k)_{[R-m+1, m)} = (x_{k'})_{[R-m+1, m)}$ since
  $(x_k)_{[R-r, m)} = (x_{k'})_{[R-r, m)}$ and $r \geq m.$
\end{proof}

In the rest of this section, we only consider the case when the slider
relation $F$ that $\chi$ defines is a function.

Next we analyze numbers of representations. We call a representation
$(x,i)$ of a pair $(y,z)$ simply a representation of configuration
$y$, because $z=f(y)$ is determined by $y$. Let $R(y,i)$ be the set of
configurations $x$ such that $(x,i)$ is a representation of $y$. By
Lemma~\ref{lem:repr} the elements of $R(y,i)$ have the form
$x=f(y)_{\minftytoi{i}} \cdot w \cdot y_{\itoinfty{i+n}}$ for some
word $w\in \ST^n$ where $n$ is the block length of $\chi$.

By Lemma~\ref{lem:reprlemma} the cardinality of the set $R(y,i)$ is
independent of $i$. Let us denote by $N(y)$ this cardinality. It turns
out that the number is also independent of the configuration $y$.

\begin{lemma}
  \label{lem:NIsSame}
  $N(y)=N(y')$ for all configurations $y,y'$.
\end{lemma}

\begin{proof}
  Let $n$ be the block length of rule $\chi$.
  \begin{enumerate}
  \item[(i)] Assume first that $y,y'$ are left-asymptotic. There is an
    index $i\in\Z$ such that $y_{\minftytoi{i}} = y'_{\minftytoi{i}}$.
    Then for any $z$ we have that
    $z_{\minftytoi{i}}y_{\itoinfty{i}}\in R(y,i-n)$ if and only if
    $z_{\minftytoi{i}}y'_{\itoinfty{i}}\in R(y',i-n)$. This gives a
    bijection between $R(y,i-n)$ and $R(y',i-n)$ so that
    $N(y)=\card{R(y,i-n)}=\card{R(y',i-n)}=N(y')$.

  \item[(ii)] Assume then that $y,y'$ are right-asymptotic. Also
    $f(y)$ and $f(y')$ are right-asymptotic so there is an index
    $i\in\Z$ such that $f(y)_{\itoinfty{i}} = f(y')_{\itoinfty{i}}$.
    Consider $z_{[i,\infty)}$ be such that
    $x=f(y)_{\minftytoi{i}}z_{\itoinfty{i}}\in R(y,i)$. Then
    $\chi^{i+}(x) = f(y)$. Consider then
    $x'=f(y')_{\minftytoi{i}}z_{\itoinfty{i}}$ obtained by replacing
    the left half $f(y)_{\minftytoi{i}}$ by $f(y')_{\minftytoi{i}}$.
    Because $f(y)_{\itoinfty{i}} = f(y')_{\itoinfty{i}}$ we have that
    $\chi^{i+}(x') = f(y')$. The configuration $y''$ represented by
    $(x',i)$ is right-asymptotic with $y'$ and satisfies
    $f(y'')=f(y')$. Because $f$ is left-closing by
    Lemma~\ref{lem:a-rule->lc}, we must have $y''=y'$. We conclude
    that $f(y)_{\minftytoi{i}}z_{\itoinfty{i}}\in R(y,i)$ implies that
    $f(y')_{\minftytoi{i}}z_{\itoinfty{i}}\in R(y',i)$, and the
    converse implication also holds by a symmetric argument. As in
    (i), we get that $N(y)=\card{R(y,i)}=\card{R(y',i)}=N(y')$.

  \item[(iii)] Let $y,y'$ be arbitrary. Configuration
    $y''=y_{\minftytoi{0}}y'_{\itoinfty{0}}$ is left-asymptotic with
    $y$ and right-asymptotic with $y'$. By cases (i) and (ii) above we
    have $N(y)=N(y'')=N(y')$.
  \end{enumerate}
\end{proof}

\noindent
As $N(y)$ is independent of $y$ we write $N$ for short.

Next we define right stairs. They were defined in~\cite{blockCA} for
reversible cellular automata -- here we generalize the concept to
other CA and show that the concept behaves well when the cellular
automaton is left-closing. A right stair is a pair of words that can
be extracted from two consecutive configurations $x$ and $f(x)$ that
coincide with $y$ and $z$, respectively, as shown in
Figure~\ref{fig:stair}. The precise definition is as follows.

\begin{figure}
  \begin{center}
    \includegraphics[scale=0.3]{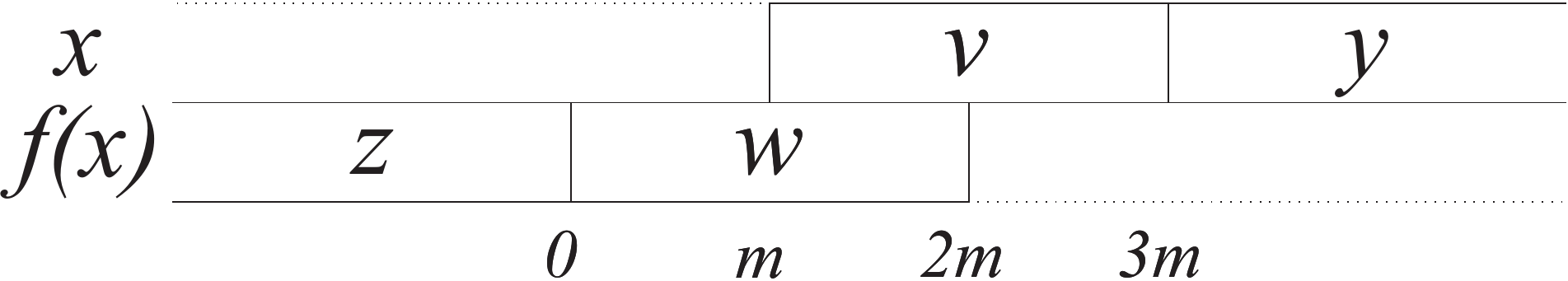}
  \end{center}
  \caption{A right stair $(v,w)$ of length $3m$ connecting $y$ and
    $z$, confirmed by $x$ at position $i=0$.}
  \label{fig:stair}
\end{figure}

\begin{definition}
  \label{def:stair}
  Let $f:\ST^\Z\longrightarrow \ST^\Z$ be a cellular automaton, and
  let $m$ be a positive integer. Let $y\in\ST^{\itoinfty{i+3m}}$ be a
  right infinite word and let $z\in\ST^{\minftytoi{i}}$ be a
  left-infinite word.
  \begin{itemize}
  \item A pair of words $(v,w)\in\ST^{2m}\times \ST^{2m}$ is a
    \emph{right stair connecting $(y,z)$} if there is a configuration
    $x\in\ST^{\Z}$ such that $vy=x_{\itoinfty{i+m}}$ and
    $zw=f(x)_{\minftytoi{i+2m}}$.
  \item The stair has \emph{length} $3m$ and it is \emph{confirmed}
    (at position $i$) by configuration $x$.
  \item We write $\stairs_{3m}(y,z)$ for the set of all right stairs
    of length $3m$ connecting $(y,z)$.
  \item We write $\stairs_{3m}$ for the union of $\stairs_{3m}(y,z)$
    over all $y$ and $z$.
  \end{itemize}
\end{definition}
Due to shift invariance, $x$ confirms $(v,w)\in\stairs_{3m}(y,z)$ if
and only if $\sigma(x)$ confirms
$(v,w)\in\stairs_{3m}(\sigma(y),\sigma(z))$. This means that
$\stairs_{3m}(y,z)=\stairs_{3m}(\sigma(y),\sigma(z))$, so it is enough
to consider $i=0$ in Definition~\ref{def:stair}. In terms of
cylinders, $(v,w)\in\stairs_{3m}$ if and only if
$f([v]_{[m,3m)})\cap [w]_{[0,2m)} \neq\emptyset$.

\begin{figure}
  \begin{center}
    \includegraphics[scale=0.3]{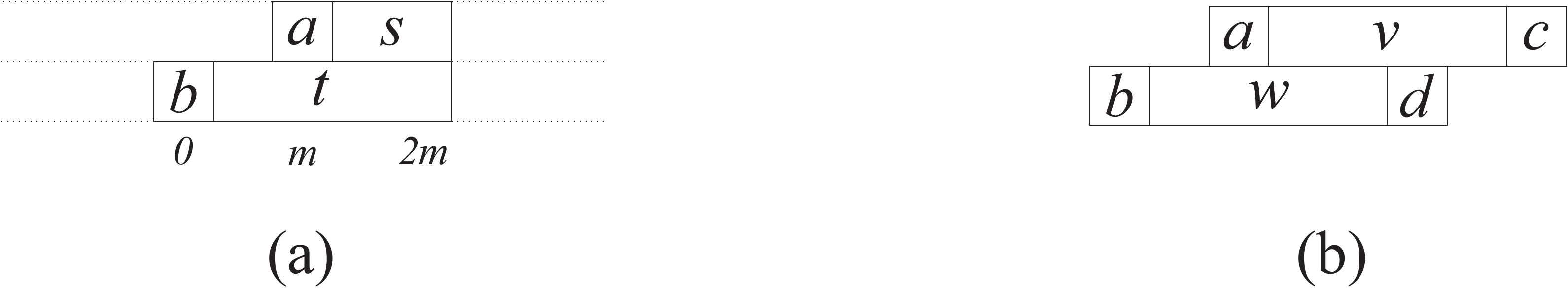}
  \end{center}
  \caption{(a) An illustration for Lemma~\ref{lem:kurka}, and (b) an
    illustration for Corollary~\ref{cor:kurkacorollary}(b) and for
    Lemma~\ref{lem:lcdiv->a-rule}.}
  \label{fig:kurka}
\end{figure}

We need the following known fact concerning left-closing CA. It
appears as Proposition 5.44 in~\cite{kurkabook} where it is stated for
right-closing CA. See Figure~\ref{fig:kurka}(a) for an illustration.

\begin{lemma}[Proposition 5.44 in~\cite{kurkabook}]
  \label{lem:kurka}
  Let $f$ be a left-closing CA. For all sufficiently large $m\in\N$,
  if $s\in \ST^m$ and $t\in \ST^{2m}$ are such that
  $f([s]_{(m,2m]})\cap [t]_{(0,2m]} \neq\emptyset$ then for all
  $b\in \ST$ there exists a unique $a\in \ST$ such that
  $f([as]_{[m,2m]})\cap [bt]_{[0,2m]} \neq\emptyset$.
\end{lemma}
The condition $f([s]_{(m,2m]})\cap [t]_{(0,2m]} \neq\emptyset$ is just
a way to write that there exists $x\in\ST^\Z$ with $x_{(m,2m]}=s$ and
$f(x)_{(0,2m]}=t$. Note that the statement of the lemma has two parts:
the existence of $a$ and the uniqueness of $a$. We need both parts in
the following.

A number $m$ is a \emph{strong\footnote{The word `strong' is added to
    distinguish this from the weaker closing radius obtained directly
    from the definition by a compactness argument.} left-closing
  radius} for a CA $f$ if it satisfies the claim of
Lemma~\ref{lem:kurka}, and furthermore $m\geq 2r$ where $r\geq 1$ is a
neighborhood radius of $f$. Next we state corollaries of the previous
lemma, phrased for right stairs in place of $s$ and $t$ to be directly
applicable in our setup.

\begin{corollary}
  \label{cor:kurkacorollary}
  Let $f$ be a left-closing CA. Let $m$ be a strong left-closing
  radius.
  \begin{itemize}
  \item[(a)] $\stairs_{3m}(y,z)=\stairs_{3m}$ for all $y$ and $z$.
  \item[(b)] Let $(vc,wd)\in\stairs_{3m}$ for $c,d\in\ST$ and
    $v,w\in \ST^{2m-1}$. For every $b\in \ST$ there exists a unique
    $a\in \ST$ such that $(av,bw)\in\stairs_{3m}$. (See
    Figure~\ref{fig:kurka}(b) for an illustration.)
  \item[(c)] Every $(v,w)\in\stairs_{3m}(y,z)$ is confirmed by a
    unique $x$.
  \end{itemize}
\end{corollary}

\begin{proof}
  (a) Let $y,y'\in\ST^{\itoinfty{3m}}$ and
  $z,z'\in\ST^{\minftytoi{0}}$ be arbitrary. It is enough to prove
  that $\stairs_{3m}(y',z')\subseteq \stairs_{3m}(y,z)$. The claim
  then follows from this and shift invariance
  $\stairs_{3m}(y,z)=\stairs_{3m}(\sigma(y),\sigma(z))$.

  First we show that
  $\stairs_{3m}(y',z')\subseteq \stairs_{3m}(y,z')$. Let
  $(v,w)\in\stairs_{3m}(y',z')$ be arbitrary, so that there exists
  $x'\in [vy']_{\itoinfty{m}}$ such that
  $f(x')_{\minftytoi{2m}} = z'w$. Then $(v,w)\in\stairs_{3m}(y,z')$ is
  confirmed by the configuration $x''$ such that
  $x''_{\minftytoi{3m}}=x'_{\minftytoi{3m}}$ and
  $x''_{\itoinfty{3m}}=y$. Indeed, $x''_{\itoinfty{m}}=vy$, and
  because $m\geq r$, the radius of the local rule of $f$, we also have
  $f(x'')_{\minftytoi{2m}}=f(x')_{\minftytoi{2m}} = z'w$.

  Next we show that $\stairs_{3m}(y,z')\subseteq \stairs_{3m}(y,z)$.
  Let $(v,w)\in\stairs_{3m}(y,z')$. We start with finite extensions of
  $w$ on the left: we prove that for every finite word $u\in\ST^*$ we
  have $f([vy]_{[m,\infty)})\cap [uw]_{[-|u|,2m)}\neq \emptyset$.
  Suppose the contrary, and let $bu\in\ST^{k+1}$ be the shortest
  counterexample, with $b\in\ST$ and $u\in\ST^{k}$. (By the
  assumptions, the empty word is not a counterexample.) By the
  minimality of $bu$, there exists $x^r\in [vy]_{[m,\infty)}$ such
  that $f(x^r)_{[-k,2m)}=uw$. Choose $s=x^r_{[-k+m,-k+2m)}$ and
  $t=f(x^r)_{[-k,-k+2m)}$ and apply the existence part of
  Lemma~\ref{lem:kurka}. By the lemma, there exists a configuration
  $x^l$ such that $x^l_{[-k+m,-k+2m)}=x^r_{[-k+m,-k+2m)}$ and
  $f(x^l)_{[-k-1,-k+2m)}=b\cdot f(x^r)_{[-k,-k+2m)}$.

  Consider $x$ obtained by gluing together the left half of $x^l$ and
  the right half of $x^r$: define
  $x_{(-\infty,-k+2m)}=x^l_{(-\infty,-k+2m)}$ and
  $x_{[-k+m,\infty)}=x^r_{[-k+m,\infty)}$. Note that in the region
  $[-k+m,-k+2m)$ configurations $x^l$ and $x^r$ have the same value.
  By applying the local rule of $f$ with radius $r$ we also get that
  $f(x)_{(-k-1,-k+2m-r)}=f(x^l)_{(-k-1,-k+2m-r)}=b\cdot
  f(x^r)_{[-k,-k+2m-r)}$ and
  $f(x)_{[-k+m+r,2m)} =f(x^r)_{[-k+m+r,2m)}$. Because $m\geq 2r$ we
  have $-k+2m-r\geq -k+m+r$, so that
  $f(x)_{(-k-1,2m)}=b\cdot f(x^r)_{(-k,2m)}=buw$. We also have
  $x_{[m,\infty)}=x^r_{[m,\infty)}=vy$, so that $x$ proves that $bu$
  is not a counterexample.

  Consider then the infinite extension of $w$ on the left by $z$:
  Applying the finite case above to each finite suffix of $z$ and by
  taking a limit, we see with a simple compactness argument that there
  exists $x\in [vy]_{[m,\infty)}$ such that $f(x)_{[-\infty,2m)}=zw$.
  This proves that $(v,w)\in\stairs_{3m}(y,z)$.

  \medskip
  
  \noindent
  (b) Let $(vc,wd)\in\stairs_{3m}$ and let $b\in\ST$ be arbitrary. Let
  $y\in\ST^{\itoinfty{3m}}$ be arbitrary, and and let
  $z\in\ST^{\minftytoi{0}}$ be such that $z_{-1}=b$. By (a) we have
  that $(vc,wd)\in\stairs_{3m}(y,z)$. Let $x$ be a configuration that
  confirms this, so $x_{[m,\infty)}=vcy$ and
  $f(x)_{(-\infty,2m)}=zwd$. Let $a=x_{m-1}$. Because
  $x_{[m-1,3m-1)}=av$ and $f(x)_{[-1,2m-1)}=bw$, configuration $x$
  confirms (at position $i=-1$) that $(av,bw)\in\stairs_{3m}$.

  Let us prove that $a$ is unique. Suppose that also
  $(a'v,bw)\in\stairs_{3m}$. We apply the uniqueness part of
  Lemma~\ref{lem:kurka} on $s$ and $t$ where $t=wd$ and $s$ is the
  prefix of $v$ of length $m$. Because $(a'v,bw)$ is a right stair,
  $f([a'v]_{[m-1,3m-1)})\cap [bw]_{[-1,2m-1)}\neq\emptyset$. Because
  $m-1\geq 2r-1\geq r$, the local rule of $f$ assigns $f(x)_{2m-1}=d$
  for all $x\in [a'v]_{[m-1,3m-1)}$, so that
  $f([a'v]_{[m-1,3m-1)})\cap [bwd]_{[-1,2m)}\neq\emptyset$. But then
  $f([a's]_{[m-1,2m)})\cap [bt]_{[-1,2m)}\neq\emptyset$, so that by
  Lemma~\ref{lem:kurka} we must have $a'=a$.

  \medskip
  
  \noindent
  (c) Suppose $x\neq x'$ both confirm that
  $(v',w')\in\stairs_{3m}(y,z)$. Then
  $x_{\itoinfty{m}}=v'y=x'_{\itoinfty{m}}$. Let $k<m$ be the largest
  index such that $x_{k}\neq x'_{k}$. Extract $a,a',b,c,d\in\ST$ and
  $v,w\in\ST^{2m-1}$ from $x$ and $x'$ as follows: $avc=x_{[k,k+2m]}$
  and $a'vc=x'_{[k,k+2m]}$ and
  $bwd=f(x)_{[k-m,k+m]}=f(x')_{[k-m,k+m]}$. Then
  $(vc,wd)\in\stairs_{3m}$ and $(av,bw), (a'v,bw)\in\stairs_{3m}$.
  This contradicts (b).
\end{proof}

Now we can prove another constraint on \arule s.

\begin{lemma}
  \label{lem:a-rule->div}
  Let $f$ be a \arule. Let $m$ be a strong left-closing radius, and
  big enough so that $f$ is defined by a bijective block rule
  $\chi:\ST^{n}\longrightarrow \ST^{n}$ of block length $n=3m$. Let
  $N$ be the number of representatives of configurations (independent
  of the configuration) with respect to $\chi$. Then
  $$N\cdot \card{\stairs_n}=\card{\ST}^n.$$
  In particular, $\card{\stairs_n}$ divides $\card{\ST}^n$.
\end{lemma}

\begin{proof}
  Fix any $y\in\ST^{[3m,\infty)}$ and $z\in\ST^{(-\infty,0)}$. Denote
  $A=\{x\in \ST^\Z\ |\ x_{[3m,\infty)}=y \mbox{ and }
  f(x)_{(-\infty,0)}=z\}$. Consider the function
  $A\longrightarrow \stairs_{3m}(y,z)$ defined by
  $x\mapsto (x_{[m,3m)}, f(x)_{[0,2m)})$. It is surjective by the
  definition of $\stairs_{3m}(y,z)$, and it is injective by
  Corollary~\ref{cor:kurkacorollary}(c). Because
  $\stairs_{3m}(y,z)=\stairs_{3m}$ by
  Corollary~\ref{cor:kurkacorollary}(a), we see that
  $\card{A}=\card{\stairs_{3m}}$.

  For each $w\in \ST^{3m}$ define configuration $x^w= zwy$.
  Representations $(x,0)$ of $y\in A$ are precisely $(x^w,0)$ for
  $w\in \ST^{3m}$. Because each $y$ has $N$ representations and there
  are $\card{\ST}^{3m}$ words $w$ we obtain that
  $N\cdot \card{\stairs_{3m}}=\card{\ST}^{3m}$.
\end{proof}

Now we prove the converse: the constraints we have proved for \arule s
are sufficient. This completes the characterization of sliders.

\begin{lemma}
  \label{lem:lcdiv->a-rule}
  Let $f$ be a left-closing cellular automaton, let $m$ be a strong
  left-closing radius, and assume that $\card{\stairs_n}$ divides
  $\card{\ST}^n$ for $n=3m$. Then $f$ is a \arule.
\end{lemma}

\begin{proof}
  Let $N=\card{\ST}^n/\card{\stairs_n}$ and pick an arbitrary
  bijection $\pi:\stairs_n\times\{1,2,\dots ,N\}\longrightarrow\ST^n$.
  Let $\loc : \ST^{2m+1}\longrightarrow\ST$ be the local rule of
  radius $m$ for the cellular automaton $f$.

  Let us define a block rule $\chi:\ST^{n+1}\longrightarrow \ST^{n+1}$
  as follows (see Figure~\ref{fig:kurka}). Consider any $c\in\ST$, any
  $k\in\{1,2,\dots ,N\}$ and any $(av,bw)\in \stairs_n$ where
  $a,b\in\ST$ and $v,w\in\ST^{2m-1}$. Let $d = \loc(avc)$. We set
  $\chi: \pi((av,bw),k)\cdot c \mapsto b\cdot \pi((vc,wd),k))$. This
  completely defines $\chi$, but to see that it is well defined we
  next show that $(vc,wd)$ is a right stair. By
  Corollary~\ref{cor:kurkacorollary}(a) we have that
  $(av,bw)\in\stairs_n(cy,z)$ for arbitrary $y,z$ so there is a
  configuration $x$ such that $x_{[m,\infty)}=avcy$ and
  $f(x)_{(\infty,2m)}=zbw$. The local rule $\loc$ determines that
  $f(x)_{2m}=\loc(avc)=d$. It follows that
  $(vc,wd)\in \stairs_n(y,zb)$, confirmed by $x$ at position $i=1$.

  Now that we know that $\chi$ is well defined, let us prove that
  $\chi$ is a bijection. Suppose $\pi((av,bw),k)\cdot c$ and
  $\pi((a'v',b'w'),k')\cdot c'$ have the same image
  $b\cdot \pi((vc,wd),k))=b'\cdot \pi((v'c',w'd'),k'))$. We clearly
  have $b=b'$, and because $\pi$ is a bijection, we have $v=v'$,
  $c=c'$, $w=w'$, $d=d'$ and $k=k'$. By
  Corollary~\ref{cor:kurkacorollary}(a) we also have that $a=a'$.

  As $\chi$ is a bijective block rule, it defines a \arule\ relation
  $F$. We need to prove that for every configuration $y$, the only $z$
  such that $(y,z)\in F$ is $z=f(y)$. Therefore, consider an arbitrary
  representation $(x,i)$ of $(y,z)\in F$. Write
  $x=z_{(-\infty,i)}\cdot \pi((av,bw),k)\cdot c \cdot
  y_{[i+n+1,\infty)}$ for letters $a,c,b\in\ST$ words
  $v,w\in \ST^{2m-1}$ and $k\in\{1,2,\dots ,N\}$. This can be done and
  as $\pi$ is surjective and all items in this representation are
  unique as $\pi$ is injective. We have $(av,bw)\in\stairs_n(cy,z)$ so
  by Corollary~\ref{cor:kurkacorollary}(c) there is a unique
  configuration $x'$ that confirms this. Then
  $x'_{[i+m,\infty)} = avc\cdot y_{[i+n+1,\infty)}$ and
  $f(x')_{(-\infty, i+2m)} = z_{(-\infty,i)}\cdot bw$. Associate $x'$
  to $(x,i)$ by defining $g(x,i)=x'$.

  Let us show that $g(\chi^i(x),i+1)=g(x,i)$. By the definition of
  $\chi$ we have
  \[
    \chi^i(x)=z_{(-\infty,i)}\cdot b\cdot \pi((vc,wd),k))\cdot y_{[i+n+1,\infty)}
  \]
  where $d=\loc(avc)$. To prove that $g(\chi^i(x),i+1)=x'=g(x,i)$ it
  is enough to show that $x'$ confirms $(vc,wd)\in\stairs_n(y,zb)$.
  But this is the case because
  $x'_{[i+m+1,\infty)} = vc\cdot y_{[i+n+1,\infty)}$ and
  $f(x')_{(-\infty, i+2m+1)} = z_{(-\infty,i)}\cdot bwd$. The fact
  that $f(x')_{i+2m}=d$ follows from $x'_{[i+m,i+3m]}=avc$ and
  $d=\loc(avc)$.

  By induction we have that for any $j\geq i$ holds
  $g(\chi^{[i,j)}(x),j)=x'$. Moreover, pair $(\chi^{[i,j)}(x),j)$
  represents the same $(y,z)\in F$ as $(x,i)$. Therefore,
  $x'_{[j+n+1,\infty)}=y_{[j+n+1,\infty)}$ and
  $f(x')_{(-\infty, j)} = z_{(-\infty,j)}$ for all $j\geq i$. Let us
  look into position $p=i+n+m+1$. Using any $j>p$ we get $f(x')_p=z_p$
  and using $j=i$ we get $x'_{[p-m,p+m]}=y_{[p-m,p+m]}$. This means
  that $z_p=\loc(y_{[p-m,p+m]})$, that is, $z_p=f(y)_p$. Because $i$
  was arbitrary, $p$ is arbitrary. We have that $z=f(y)$, which
  completes the proof.
\end{proof}

\begin{theorem}
\label{thm:SliderCharacterization}
The function $f$ admits a slider if and only if $f$ is a left-closing
cellular automaton and $\card{\stairs_n}$ divides $|\ST|^n$ for
$n = 3m$ where $m$ is the smallest strong left-closing radius.
\end{theorem}

We can state this theorem in a slightly more canonical (but completely
equivalent) form by normalizing the length of stairs:

By Corollary~\ref{cor:kurkacorollary}, for a left-closing cellular
automaton $f$ the limit
\begin{equation}
  \label{eq:lambda}
  \lambda_f = \lim_{m\to\infty} \frac{\card{\stairs_{3m}}}{\card{\ST}^{3m}}
\end{equation}
is reached in finite time, namely as soon as $m$ is a left-closing
radius, and thus $\lambda_f$ is rational for left-closing $f$. In
\cite{blockCA} it is shown that the map $f \mapsto \lambda_f$ gives a
homomorphism from the group of reversible cellular automata into the
rational numbers under multiplication. For a prime number $p$ and an
integer $n$, write $v_p(n)$ for the largest exponent $k$ such that
$p^k | n$. For prime $p$ and rational $r = m/n$, write
$v_p(r) = v_p(m) - v_p(n)$ for the \emph{$p$-adic valuation} of $r$.

\begin{theorem}
  \label{thm:SliderCharacterization2}
  The function $f$ is a slider if and only if $f$ is a
  left-closing cellular automaton and $v_p(\lambda_f) \leq 0$ for all
  primes $p$.
\end{theorem}

\begin{example}
  Let $A = \{0,1\} \times \{0,1,2\}$ and write $\sigma_2$ and
  $\sigma_3$ for the left shifts on the two tracks of $A^\Z$. Then
  consider $f = \sigma_2 \times \sigma_3^{-1}$. For this CA we have by
  a direct computation $|\stairs_3| = 2^2 \cdot 3^4$ so
  $\lambda_f = 2^2 \cdot 3^4 / 6^3$ so $v_3(\lambda_f) = 1 > 0$, and
  thus $f$ is not a slider. Similarly we see that
  $\sigma_3 \times \sigma_2^{-1}$ is not a slider.
\end{example}

\begin{example}
  \label{ex:LeftXOR}
  Let $S=\{0,1\}$ and consider the exclusive-or CA with neighborhood
  $\{-1,0\}$, i.e. $f(x) = x + \sigma^{-1}(x)$. Then $f$ is
  left-closing but a direct computation shows
  $v_2(\lambda_f) = 1 > 0$, so $f$ is not a slider. Compare with
  Example~\ref{ex:RightXOR}.
\end{example}

\subsection{Definition of \brule s}

An alternative approach not requiring bijectivity of $\chi$ is
specified in the following:

\begin{definition}
  \label{def:brelation}
  A block rule $\chi$
  defines a \emph{\brule\ relation} $F \subset \ST^\Z \times \ST^\Z$ by
  $(y, z) \in F$ iff some subsequence of $\chi^{0+}(y), \chi^{-1+}(y), \chi^{-2+}(y),\dots$ converges to $z$.
\end{definition}

\begin{lemma}
  \label{lem:SweepingBasics}
  The projection $(y,z)\mapsto y$ on the first component maps a
  \brule\ relation $F$ surjectively onto $\ST^\Z$. The relation $F$ is
  a function $f$ if and only if for each configuration $y$ the limit
  $\lim_{i\to-\infty} \toright{\chi}(y)$ exists and equals $f(y)$.
\end{lemma}

\begin{proof}
  For every $y\in\ST^\Z$ the sequence
  $\chi^{0+}(y), \chi^{-1+}(y), \chi^{-2+}(y),\dots$ has a converging
  subsequence with some limit $z$. Then $(y,z)\in F$ so the projection
  is onto.

  If $z=\lim_{i\to-\infty} \toright{\chi}(y)$ exists then every
  subsequence of $\chi^{0+}(y), \chi^{-1+}(y), \chi^{-2+}(y),\dots$
  converges to $z$ so $z$ is the unique configuration such that
  $(y,z)\in F$. Conversely, if $\lim_{i\to-\infty} \toright{\chi}(y)$
  does not exist then
  $\chi^{0+}(y), \chi^{-1+}(y), \chi^{-2+}(y),\dots$ has two
  subsequences converging to distinct $z_1$ and $z_2$. In this case
  $(y,z_1)$ and $(y,z_2)$ are both in relation $F$. 
\end{proof}

\begin{definition}
  \label{def:brule}
  Let $\chi$ be a block rule such that for each configuration $y$ the
  limit $z= \lim_{i\to-\infty} \toright{\chi}(y)$ exists. The function
  $y\mapsto z$ is called the \emph{\brule}\ defined by~$\chi$.
\end{definition}

Before we are going to compare the notions of sliders and sweepers we
provide a result on a special kind of Mealy automata.

% -----------------------------------------------------------------------
\subsection{A note on finite Mealy automata}
\label{subsec:mealy}

In this section we consider Mealy automata with a set $Q$ of states
and where the set $A$ of input symbols and the set of output symbols
coincide.
For convenience instead of pairs of elements we use words of length $2$.
Thus, we denote by $\mealy\from Q A \to A Q$ the function mapping the
current state $q$ and an input symbol $a$ to $\mealy(qa)=a'q'$, where
$q'$ is the new state of the automaton.

The motivation for this is the following.
When a block rule $\chi$ is sweeping over a configuration one can think
of the block $q\in \ST^n$ where $\chi$ will be applied next as
encoding the state of a Mealy automaton.
The word $a\in\ST^n$ immediately to the right of it is the next input
symbol.
By applying $\chi$ at positions $0,1,\dots,n-1$ the word $qa$ is
transduced into a word $a'q'\in\ST^{2n}$ where $a'$ can be
considered the output symbol and $q'$ the next state of the automaton.
When $\chi$ is bijective then clearly $\mealy$ is bijective, too.

Let $\delta:QA\to Q$ denote the function yielding only the new state
of the Mealy automaton.
The extension $\delta^*\from QA^*\to A^*Q$ to input \emph{words} is
for all states $q$, all inputs $w\in A^*$ and $a\in A$ defined by
$\delta^*(q\varepsilon)=q$ and $\delta^*(qwa)=\delta(\delta^*(qw)a)$.

Because of the application we have in mind we now restrict ourselves
to the case where $Q=A$ and speak of elements $e\in Q$.
Let $\bar{e}=(\ldots,e_{-2},e_{-1},e_0)$ denote a sequence of elements
which is infinite to the left.

\begin{definition}
  A finite tail $\bar{e}_i=(e_{-i},\ldots,e_0)$ of $\bar{e}$ is
  \emph{good for $q$} if $\delta^*(\bar{e}_i)=q$.
  An infinite sequence $\bar{e}$ is \emph{good for $q$} if infinitely
  many finite tails $\bar{e}_i=(e_{-i},\ldots,e_0)$ are good for $q$.

  A \emph{state $q$ is good}, if there is an infinite sequence
  $\bar{e}$ that is good for $q$.
  Let $G\subseteq Q$ denote the set of good states and $B\subseteq Q$
  the set of bad states.
\end{definition}

\begin{lemma}
  \label{lem:mealy}
  If $\mealy$ is bijective then $G=Q$ and $B=\emptyset$.
\end{lemma}

\begin{proof}
  First, observe that the property of being good is preserved by
  $\delta$. If $g$ is good, then each $\delta(ga)$ is good, too:
  If $\bar{e}$ is good for $g$, then $\bar{e}a$ is good for
  $\delta(ga)$ since $\delta^*(e_{-i},\ldots,e_0)=g$ implies
  $\delta^*(e_{-i},\ldots,e_0,a)=\delta(ga)$.
  This means that $\mealy(GA)\subseteq AG$.

  Since $\mealy$ is injective and $\card{GA}=\card{AG}$, in fact
  $\mealy(GA)= AG$.
  Therefore $\mealy(BA)\subseteq AB$, that is $\delta$ preserves bad
  states.
  Now, assume that there indeed exists a bad state $b\in B$.
  Consider $\bar{b}=(\ldots, b,b,b)$.
  The states $b_i=\delta^*(b^i)$ are all bad, but at least one of them
  happens infinitely often, which would mean that it is good.
  Contradiction.
\end{proof}

% -----------------------------------------------------------------------
\subsection{Relation between sliders and sweepers}
\label{subsec:a-rule-b-rule}

Compared to definition~\ref{def:arule} the advantage of
definition~\ref{def:brule} is that it does not require $\chi$ to be
bijective.
But as long as $\chi$ is bijective, there is in fact no difference.

\begin{theorem}
\label{thm:SweeperCharacterization}
  Let $\chi$ be a bijective block rule and $f$ a one-dimensional CA.
  The slider relation defined by $\chi$ is equal to $f$ if and only if
  the sweeper relation it defines is equal to $f$.
\end{theorem}

The two implications are considered separately in
Lemmata~\ref{lem:not-b=>not-a} and~\ref{lem:not-a=>not-b} below.
For the remainder of this section let $\chi\from \ST^n \to\ST^n $
always denote a bijective block rule and let
$f\from\ST^{\Z}\to\ST^{\Z}$ denote a one-dimensional CA (without
stating this every time).

\begin{lemma}
  \label{lem:not-b=>not-a}
  If $\chi$ is not a sweeper for $f$ then it is not a slider for
  $f$.
\end{lemma}

\begin{proof}
  If $\chi$ is not a sweeper for $f$ then there is a configuration
  $y$ for which the limit $\lim_{i\to-\infty} \chi^{\itoinfty{i}}(y)$
  does not exist or is wrong.
  In both cases there is a cell $j\in \Z$ and a state $s\in\ST$ such
  that $s\not=f(y)_j$ but $\chi^{\itoinfty{i}}(y)_j=s$ for infinitely
  many $i < j-n$.

  We will construct a configuration $x$ such that $\chi^{j-}(x)=y$
  and $\chi^{j+}(x)_j=s\not=f(y)_j$.
  Therefore $\chi$ is not a slider for $f$ (see Def.~\ref{def:arule}).

  As a first step we subdivide the ``left part'' $(-\infty,j+n)$ of
  $\Z$ into \emph{windows} $W_k$ of length $n$.
  For $k\geq 0$ let $p_k=-kn+j$ denote the smallest index in $W_k$,
  i.\,e.~$W_k=[p_k,p_{k-1})$ (where $p_{-1}=j+n$).
  Analogously divide the ``left part'' of $y$ into words $y^{(k)}$ of
  length $n$ by setting $y^{(k)} = y|_{W_k}$ (see Fig.~\ref{fig:Wk}).

    \begin{figure}[ht]
    \centering
    \begin{tikzpicture}[x={(1mm,0mm)},y={(0mm,1mm)},decoration={brace,amplitude=2mm}]
      \draw (-3,0) node {$y$};
      \draw (0,0) -- +(100,0);
      \foreach \x in {2,4,...,99} {
        \draw (\x,0) -- +(0,1) -- +(0,-1);
      }
      \foreach \x in {10,30,...,90} {
        \draw[thick] (\x,0) -- +(0,2) -- +(0,-2);
      }
      \draw (71,8) node[rotate=20,anchor=south west] (N0) {$p_0=j$};
      \draw[->] (N0.south west) -- (71,2);
      \draw (51,8) node[rotate=20,anchor=south west] (N1) {$p_1=-n+j$};
      \draw[->] (N1.south west) -- (51,2);
      \draw (31,8) node[rotate=20,anchor=south west] (N2) {$p_2=-2n+j$};
      \draw[->] (N2.south west) -- (31,2);

      \draw [decorate] (90,-3) -- (70,-3) (80,-7) node {$W_0$};
      \draw [decorate] (70,-3) -- (50,-3) (60,-7) node {$W_1$};;
      \draw [decorate] (50,-3) -- (30,-3) (40,-7) node {$W_2$};;
    \end{tikzpicture}
    \caption{For configuration $y$ the windows $W_k$ contain the words
      $y^{(k)}$.}
    \label{fig:Wk}
  \end{figure}
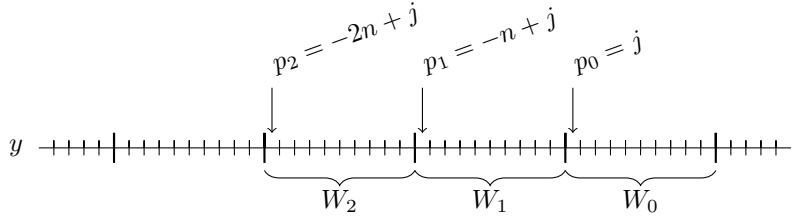

  Let $M$ denote the set $\{i \mid \chi^{\itoinfty{i}}(y)_j=s\}$.
  $M$ contains infinitely many integers $i< p_1=j-n$.
  Then there has to be a word $v^{(0)}\in\ST^n$ such that the set
  $M_0 = \{ i\in M \mid i<p_1 \text{ and } \chi^{[i,j)}(y)|_{W_0} =
  v^{(0)} \}$ is infinite.
  Since $M_0\subseteq M$ certainly $\chi(v^{(0)})_0 = s$ holds.

  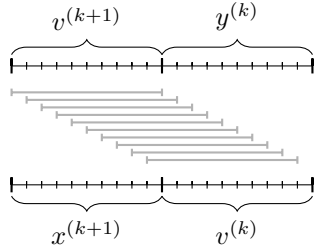
\begin{figure}[ht]
    \centering
    \begin{tikzpicture}[x={(2mm,0mm)},y={(0mm,0.5mm)},decoration={brace,amplitude=2mm}]
      % \draw (0,30) node {QQQ};
      \draw (0,0) -- +(20,0);
      \foreach \x in {0,1,...,20} {
        \draw (\x,0) -- +(0,1) -- +(0,-1);
      }
      \foreach \x in {0,10.,20} {
        \draw[thick] (\x,0) -- +(0,2) -- +(0,-2);
      }

      \draw [decorate] ( 0,3) -- node[above=2mm] {$v^{(k+1)}$} (10,3);
      \draw [decorate] (10,3) -- node[above=2mm] {$y^{(k)}$} (20,3);

      \begin{scope}[yshift=-10]]
        \foreach \x in {0,1,...,9} {
          \draw[thick,gray!60!white]  (\x,-\x) ++(0,-\x) ++(0,1) -- +(0,-2) ++(0,-1) -- ++ (10,0)
          ++(0,1) -- ++(0,-2);
        }
      \end{scope}

      \begin{scope}[yshift=-45]
      \draw (0,0) -- +(20,0);
      \foreach \x in {0,1,...,20} {
        \draw (\x,0) -- +(0,1) -- +(0,-1);
      }
      \foreach \x in {0,10.,20} {
        \draw[thick] (\x,0) -- +(0,2) -- +(0,-2);
      }

      \draw [decorate] (10,-3) -- node[below=2mm] {$x^{(k+1)}$} ( 0,-3);
      \draw [decorate] (20,-3) -- node[below=2mm] {$v^{(k)}$} (10,-3);
      \end{scope}
    \end{tikzpicture}
    \caption{Transition from $v^{(k+1)}y^{(k)}$ to $x^{(k+1)}v^{(k)}$
      by applying $\chi$ from left to right at the positions indicated
      by the gray bars. Application of the inverse $\xi$ from
      right to left realizes the opposite transition from bottom to
      top.}
    \label{fig:a-b-1}
  \end{figure}

  For all $k\geq 0$ we now inductively define words $v^{(k+1)}$ and
  $x^{(k+1)}$ (all of length $n$) and along with it infinite
  sets $M_{k+1}$.
  Since $M_k$ is infinite, there is a word $v^{(k+1)}$ such that the set
  \[
    M_{k+1} = \{ i \in M_k \mid i< p_{k+2} \text{ and }
    \chi^{[i,p_k)}(y)|_{W_{k+1}} = v^{(k+1)} \}
  \]
  is infinite.
  Since $M_{k+1}\subseteq M_k$ one has
  $\chi^{[p_{k+1},p_k)}(v^{(k+1)}y^{(k)}) = x^{(k+1)} v^{(k)}$ for
  some $x^{(k+1)}$ (see Fig.\ref{fig:a-b-1}).
  Since $\chi$ is bijective, for the inverse $\xi$ of $\chi$ holds
  $v^{(k+1)}y^{(k)} = \xi^{[p_{k+1},p_k)^R}(x^{(k+1)} v^{(k)})$, too.
  Note again that $\xi$ is applied from right to left.

  Now choose configuration
  $x= \cdots x^{(3)} \cdot x^{(2)} \cdot x^{(1)}
  \cdot v^{(0)} \cdot y_{\itoinfty{j+n}}$.

  On one hand $\chi^{j+}(x)$ already after the application of $\chi$
  at position $j$ produces state $s$ there which never changes again.
  Thus $\chi^{j+}(x)\not=f(y)$.

  On the other hand by construction for all $k\geq 0$ holds
  $\xi^{[p_{k+1},p_k)^R}(x^{(k+1)} v^{(k)}) = v^{(k+1)}y^{(k)}$.
  Therefore

  \begin{align*}
    \xi^{[p_1,p_0)^R}(x)
    &=\xi^{[p_1,p_0)^R}( \cdots x^{(3)} x^{(2)} x^{(1)} v^{(0)}y_{\itoinfty{j+n}}) \\
    &=\cdots x^{(3)} x^{(2)} v^{(1)}z^{(0)}y_{\itoinfty{j+n}} \\
    \intertext{and by induction for all $k\geq 0$}
    \xi^{[p_{k+1},p_0)^R}(x)
    &= \xi^{[p_{k+1},p_0)^R}( \cdots x^{(3)} x^{(2)} x^{(1)} v^{(0)}y_{\itoinfty{j+n}}) \\
    &= \cdots x^{(k+1)} v^{(k)} y^{(k-1)}\cdots y^{(0)}y_{\itoinfty{j+n}}
  \end{align*}
  Obviously one gets $\toleft{\xi}(x) =y$.
\end{proof}

\begin{lemma}
  \label{lem:not-a=>not-b}
  If $\chi$ is not a slider for $f$ then it is not a sweeper for
  $f$.
\end{lemma}

\begin{proof}
  If $\chi$ is not a slider for $f$ %(see Def.~\ref{def:a-rule})
  then there exists a configuration $x$ and an $i\in\Z$ such that for
  $z=\toright{\chi}(x)$ and $y=\toleft{\chi}(x)$ one has
  $f(y) \not= z$.
  Let $\xi$ be the inverse of $\chi$.
  Let $j$ be a cell where $f(y)_j\not= z_j$.
  If $j<i+n$ instead of $x$ can consider
  $x'=\xi^{[i-1,\dots,i-m]} (x)$ for some sufficiently large
  $m$.
  Assume therefore that $j\geq i+n$.

  We will prove that there is a configuration $v$ such that for
  infinitely many positions $m$ the configuration
  $\chi^{\itoinfty{m}}(v)$ will not have the correct state at position
  $j$.
  Therefore the limit $\lim_{m\to-\infty}\chi^{\itoinfty{m}}(v)$
  cannot exist and have the correct state at position $j$.
  % (see Def.~\ref{def:b-rule})
  Thus $\chi$ is not a sweeper for $f$.

  Below the abbreviation $\STT=\ST^n$ is used.

  Configuration $x$ is of the form
  $z_{\minftytoi{i}}\cdot w \cdot y_{\itoinfty{i+n}}$ for some
  $w\in\STT$.
  Applying $\chi$ at position $i$ and further to the right produces
  the same result independent of what is to the left of $w$.
  Therefore if $z_{\minftytoi{i}}$ is replaced by any
  $z'_{\minftytoi{i}}$ still the wrong state is produced at position
  $j$.

  Define a Mealy automaton with $Q=A$ by
  $\mealy(qa)=\chi^{[0,n)}(qa)$ (observe that $qa\in\ST^{2n}$).
  Since $\mealy$ is bijective, one can now use the result from
  Lemma~\ref{lem:mealy} and conclude that there is a sequence
  $(\ldots,v^{(2)},v^{(1)})$, infinite to the left, of elements
  $v^{(k)}\in\STT$ such that
  \begin{equation}
    \text{$\delta^*(v^{(k)}\cdots v^{(1)})=w$ for infinitely many $k$.} \label{cond:star}
  \end{equation}
  Let $v$ be the infinite to the left half-configuration obtained by
  concatenating all $v^{(k)}$, more precisely
  $v\from \minftytoi{i+n}\to\ST$ where $v_{-kn+j+i} = v^{(k)}_j$ for
  all $k\geq 1$ and all $j\in[0,n)$.

  Condition~(\ref{cond:star}) implies that for infinitely many
  $k\geq 1$ applying $\chi$ in $v$ from position $-kn+i$ up to but
  excluding $i$ produces $w$ at the end, i.\,e.~in the window
  $[i,i+n)$. In other words
  $\chi^{[-kn+i,i)}(v) = v'_{\minftytoi{i}} \cdot w$ ($v'$
  depends on $k$ but doesn't matter).
  Therefore for infinitely many $k$
  \begin{align*}
    \chi^{[-kn+i,i)}(vy_{\itoinfty{i+n}})
    &= \cdots \cdot w\cdot z_{\itoinfty{i+n}} \\
    \text{and \quad} \chi^{\itoinfty{-kn+i}}(vy_{\itoinfty{i+n}})
    &= \cdots \cdot w'\cdot z_{\itoinfty{i+n}}
  \end{align*}
  Since we could assume that the position $j$ where $f(y)_j\not= z_j$
  is in the interval $\itoinfty{i+n}$ one can conclude that $\chi$ is
  not a sweeper for $f$.
\end{proof}

While the slider and sweeper relations defined by a block rule are
equal when at least one them defines a cellular automaton, sweeper
relations can also define non-continuous functions.

\begin{example}
  \label{ex:NotClosed}
  Let $\ST = \{\ns{0}{0},\ns{0}{1},\ns{1}{0},\ns{1}{1}\}$ and define
  $\chi : \ST^2 \to \ST^2$ by
  $\chi(\ns{1}{0} \ns{0}{0}) = \ns{0}{0} \ns{0}{1}$,
  $\chi(\ns{0}{0} \ns{0}{1}) = \ns{1}{0} \ns{0}{0}$, and
  $\chi(ab) = ab$ for
  $ab \notin \{\ns{1}{0} \ns{0}{0}, \ns{0}{0} \ns{0}{1}\}$.

  We claim that $\lim_{i \rightarrow -\infty} \chi^{i+}(x)$ is
  well-defined for all $x \in \ST^\Z$, so that the sweeper relation
  $\chi$ defines is a function. Let $x \in \ST^\Z$ be arbitrary, and
  let $n \in \Z$. We need to show that $\chi^{i+}(x)_n$ converges.

  Suppose first that for some $k < n$, we have $x_k = \ns{1}{a}$ for
  $a \in \{0,1\}$. Then for all $i < k$, the value $\chi^{i+}(x)_n$ is
  independent of the values $x_j \leq k$, since
  $\chi^{[i,k-1]}(x)_k = \ns{1}{a}$, meaning that the sweep is
  synchronized (in the sense that whatever information was coming from
  the left is forgotten and the sweep continues the same way) and
  $\chi^{i+}(x)_n$ is determined by $x_{[k,n]}$ for all $i < k$. Thus,
  in this case $\chi^{i+}(x)_n$ converges.

  Suppose then that for all $k < n$, $x_k = \ns{0}{a}$ for some
  $a \in \{0,1\}$. If $x_k = \ns{0}{0}$ for some $k < n$, then since
  $x_{k-1} \neq \ns{1}{0}$ we also have
  $\chi^{[i,k-2]}(x)_{k-1} \neq \ns{1}{0}$.
  Thus, the value at $k$ does not change when $\chi$ is applied at
  $k-1$, and as in the previous paragraph, the sweep is synchronized
  at this position. Again $\chi^{i+}(x)_n$ is determined by
  $x_{[k,n]}$ for all $i < k$, so $\chi^{i+}(x)_n$ converges.

  In the remaining case, $x_k = \ns{0}{1}$ for all $k < n$. Then since
  $\chi(\ns{0}{1} \ns{0}{1}) = \ns{0}{1} \ns{0}{1}$, the rule is not
  applied in the left tail of $x$, and thus certainly $\chi^{i+}(x)_n$
  converges.

  The function defined by the sweeper relation is not continuous at
  $\ns{0}{1}^\Z$ since $\chi^\Z(\ns{0}{1}^\Z) = \ns{0}{1}^\Z$ while
  for any $n \in \N$ we have
  \[
    \chi^\Z(...\ns{0}{0} \ns{0}{0} \ns{0}{0} \ns{0}{1}^n .
    \ns{0}{1}^\N) = ...\ns{0}{0} \ns{0}{0} \ns{1}{0} \ns{1}{0}^n .
    \ns{1}{0}^\N
  \]
\end{example}

%======================================================================
\section{Realization of bi-closing CA using LR and RL sliders}
\label{sec:rev-ca}

In the definition of a slider we use a left-to-right slide of the
window to realize the CA transition.
Of course, one can analogously define \emph{right-to-left sliders} and
prove a characterization via right-closing CA.
We can also alternate these two types of rules, and obtain a
ladder-shaped hierarchy analogous to the Borel, arithmetic and
polynomial hierarchies. %$\Sigma_n, \Pi_n, \Delta_n$

\begin{definition}
  Let $\SR$ denote the set of CA definable as slider relations with
  the ``from left to right'' as in Definition~\ref{def:arule}.
  Analogously let $\SL$ denote the set of CA definable as
  right-to-left slider relations.
  Denote $\Delta = \SL \cap \SR$.
  Let now $\SL_0 = \SR_0 = \{\ID\}$, and for all $k \in \N_0$ let
  $\SL_{k+1} = \SL \circ \SR_{k}$ and $\SR_{k+1} = \SR \circ \SL_{k}$.
  For all $n$, write $\Delta_n = \SL_n \cap \SR_n$.
\end{definition}

Note that in $\SL_n$, there are $n$ sweeps (slider applications) in
total, and the last sweep goes from right to left.
We have $\SL_1 = \SL$, $\SR_1 = \SR$, $\Delta_1 = \Delta$.
See Figure~\ref{fig:Hierarchy}. %\vspace{-0.3cm}
\begin{figure}
  \begin{center}
    \begin{tikzpicture}
      \node (I0) at (-1,1) {};
      \node[xshift=-4] (D0) at (I0) {$\Delta_1$};
      \node (L1) at (0,0) {$\SL_1$};
      \node (I1) at (1,1) {};
      \node[right = 0mm of I1] (D1) {$\Delta_2$};
      \node (R1) at (0,2) {$\SR_1$};
      \node (L2) at (2,0) {$\SL_2$};
      \node (I2) at (3,1) {};
      \node[right = 0mm of I2] (D2) {$\Delta_3$};
      \node (R2) at (2,2) {$\SR_2$};
      \node (L3) at (4,0) {$\SL_3$};
      \node (I3) at (5,1) {};
      \node[right = 0mm of I3] (D3) {$\Delta_4$};
      \node (R3) at (4,2) {$\SR_3$};
      \node (L4) at (6,0) {$\SL_4$};
      \node (I4) at (7,1) {};
      \node[right = 0mm of I4] (D4) {$\Delta_5$};
      \node (R4) at (6,2) {$\SR_4$};
      \node (L5) at (8,0) {$\cdots$};
      \node (R5) at (8,2) {$\cdots$};
      \draw (I0) -- (R1); \draw (I0) -- (L1);
      \draw (L1) -- (R2); \draw (L1) -- (L2);
      \draw (R1) -- (L2); \draw (R1) -- (R2);
      \draw (L2) -- (R3); \draw (L2) -- (L3);
      \draw (R2) -- (L3); \draw (R2) -- (R3);
      \draw (L3) -- (R4); \draw (L3) -- (L4);
      \draw (R3) -- (L4); \draw (R3) -- (R4);
      \draw (L4) -- (R5); \draw (L4) -- (L5);
      \draw (R4) -- (L5); \draw (R4) -- (R5);
    \end{tikzpicture}
  \end{center}
  \vspace{-0.5cm}
  \caption{The sliding hierarchy.}
  \label{fig:Hierarchy}
\end{figure}
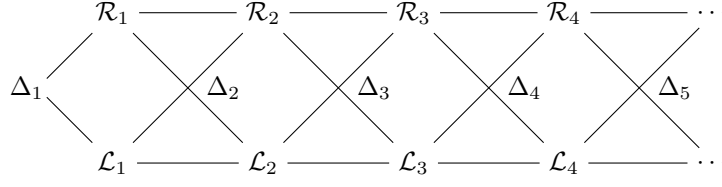

%\vspace{-0.3cm}

In Theorem~\ref{thm:slider-hierarchy=closing-hierarchy} below we will
show a close relation between this ``slider hierarchy'' and a
``closingness hierarchy'' defined as follows, exactly analogously.
Let $\CL$ denote the set of left-closing CA and $\CR$ the set of
right-closing CA. Define $\CL_0 = \CR_0 = \{\ID\}$ and for all $k$,
$\CL_{k+1} = \CL \circ \CR_k$ and $\CR_{k+1} = \CR \circ \CL_k$.

As always with such hierarchies, it is natural to ask whether they are
infinite or collapse at some finite level. We do not know if either
hierarchy collapses, but we show that after the first level, the
hierarchies coincide. The main ingredients for the theorem are the
following two lemmata.

\begin{lemma}
  \label{lem:DividesPower}
  Let $f$ be a left-closing CA. For all $n$ large enough,
  $\card{\stairs_{n}}$ divides some power of $\card{\ST}$.
\end{lemma}

\begin{proof}
  Let $m$ be a strong left-closing radius for $f$. Number $m$ can be
  chosen as large as needed. Let $\loc$ be the local update rule of
  $f$ of radius $3m$. By Theorem 14.7 in~\cite{hedlund} there exist,
  for $k=3m$ chosen sufficiently large,
  \begin{itemize}
  \item positive integers $L,M$ and $R$ such that
    $L\cdot M\cdot R = |\ST|^{2k}$,
  \item pairwise different words $u_1,\dots,u_M$ of length $k$,
  \item sets ${\cal L}_1,\dots, {\cal L}_M\subseteq \ST^{2k}$ of words
    of length $2k$, each of cardinality $\card{{\cal L}_i}=L$,
  \item sets ${\cal R}_1,\dots, {\cal R}_M\subseteq \ST^{2k}$ of words
    of length $2k$, each of cardinality $\card{{\cal R}_i}=R$,
  \item a word $w$ of length $3k$ whose set pre-images of length $5k$
    under $\loc$ is precisely
    \[
      \bigcup_{i=1}^M {\cal L}_iu_i{\cal R}_i.
    \]
  \end{itemize}
  See Figure~\ref{fig:hedlund} for an illustration.
  
  \begin{figure}
    \begin{center}
      \includegraphics[scale=0.5]{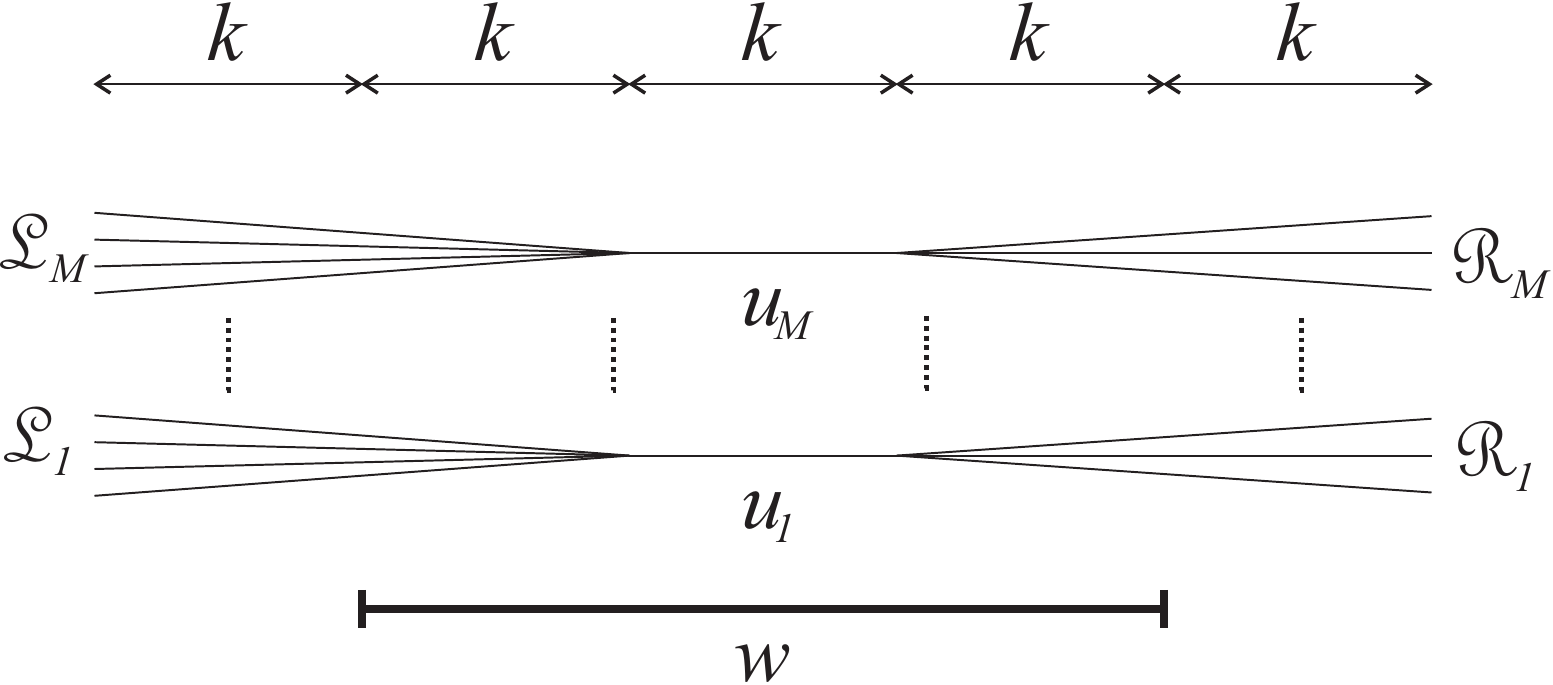}
    \end{center}
    \caption{Illustration of Theorem 14.7 in~\cite{hedlund}, with
      $L=4$, $R=3$.}
    \label{fig:hedlund}
  \end{figure}
  
  Let $y\in \ST^{[k,\infty)}$ be arbitrary and let
  $z\in\ST^{(-\infty,0)}$ be such that $z_{[-3k,0)}=w$. Define
  $A=\{x\in\ST^\Z\ |\ x_{[k,\infty)}=y, f(x)_{(-\infty,0)}=z\}$. By
  Corollary~\ref{cor:kurkacorollary} we know that
  $\card{\stairs_{3m}}=\card{A}$.

  \medskip

  \noindent
  (i) If $x\in A$ then $x_{[-4k,k)}$ is a pre-image of
  $w=z_{[-3k,0)}$. This means that for some $i\in\{1,\dots, M\}$, we
  have $x_{[-4k,k)}\in{\cal L}_iu_i{\cal R}_i$ and, in particular,
  $x_{[-2k,k)}\in u_i{\cal R}_i$.

  \medskip
  
  \noindent
  (ii) Conversely, let $i\in \{1,\dots, M\}$ and $v\in{\cal R}_i$ be
  arbitrary. Words in ${\cal L}_iu_iv$ are pre-images of $w$ so
  $f([u_iv]_{[-2k,k)})\cap [w]_{[-3k,0)} \neq\emptyset$. Because $f$
  is left-closing and $k$ is a strong left-closing radius for $f$
  there exists a unique $x\in A$ such that $x_{[-2k,k)}=u_iv$.

  \medskip
  
  \noindent
  From (i) and (ii) we can conclude that $\card{A}=M\cdot R$. Hence
  $L\cdot \card{\stairs_{3m}}=L\cdot \card{A}=L\cdot M\cdot R =
  |\ST|^{2k}$.
\end{proof}

\begin{lemma}
  \label{lem:ForLargeEnough}
  Let $f$ be a left-closing CA. Then for any large enough $n$, we have
  $\sigma^n \circ f \in \SR$.
\end{lemma}

\begin{proof}
  By the previous lemma, we have $v_p(\lambda_f) = 0$ for all
  $p \nmid |\ST|$. Similarly as in \cite{blockCA} one sees that the
  map $g \mapsto \lambda_g$ is a homomorphism among left-closing CA,
  so
  \[
    v_p(\lambda_{\sigma^n \circ f}) = v_p(\lambda_{\sigma^n} \cdot
    \lambda_{f}) = v_p(\lambda_f) - n v_p(|\ST|) \leq 0
  \]
  for large enough $n$. The claim follows from
  Theorem~\ref{thm:SliderCharacterization2}.
\end{proof}

\begin{theorem}
  \label{thm:slider-hierarchy=closing-hierarchy}
  For each $k\in\N$ with $k\geq 2$ we have $\SL_k = \CR_k$ and $\SR_k = \CL_k$.
\end{theorem}

\begin{proof}
  By Lemma~\ref{lem:a-rule->lc} we have $f \in \SL \implies f \in \CR$
  and $f \in \SR \implies f \in \CL$, so by induction
  $\SL_k \subset \CR_k$ and $\SR_k \subset \CL_k$.

  Suppose then that $f \in \CR_k$ and $k \geq 2$, so
  \[
    f = f_1 \circ f_2 \circ \cdots \circ f_{k-1} \circ f_k
  \]
  where $f_i \in \CR$ for odd $i$ and $f_i \in \CL$ for even $i$. Then
  write
  \[
    f = (f_1 \circ \sigma^{n_1}) \circ (f_2 \circ \sigma^{n_2}) \circ
    \cdots \circ (f_{k-1} \circ \sigma^{n_{k-1}}) \circ (f_k \circ
    \sigma^{n_k})
  \]
  where $\sum_{i = 1}^k n_i = 0$ and for each odd $i$, $n_i \leq 0$ is
  small enough that $f_i \circ \sigma^{n_i} \in \SL$ and for each even
  $i$, $n_i \geq 0$ is large enough that
  $f_i \circ \sigma^{n_i} \in \SR$. This shows that $f \in \SL_k$.
  Similarly $\CL_k \subset \SR_k$, concluding the proof.
\end{proof}

A cellular automaton $f$ is \emph{bi-closing} if it is both
left-closing and right-closing, i.e. $f \in \Delta^{cl}_1$. Such
cellular automata are also called \emph{open}, since they map open
sets to open sets. By the previous result, every bi-closing CA can be
realized by a left-to-right sweep followed by a right-to-left sweep by
bijective block rules:

\begin{theorem}
  Each bi-closing CA is in $\Delta_2$.
\end{theorem}

% ======================================================================
\section{Decidability}
\label{sec:Decidability}

In this section, we show that our characterization of sliders and
sweepers shows that the existence of them for a given CA is decidable.
We also show that given a block rule, whether it defines some CA as a
slider (equivalently as a sweeper) is decidable. We have seen that
sweepers can also define shift-commuting functions which are not
continuous. We show that this condition is also decidable.

\begin{lemma}
  \label{lem:decidable-left-closing-strong-radius}
  Given a cellular automaton $f : \ST^\Z \to \ST^\Z$, it is decidable
  whether it is left-closing, and when $f$ is left-closing, a strong
  left-closing radius can be effectively computed.
\end{lemma}

\begin{proof}
  It is obviously decidable whether a given $m \in \N$ is a strong
  left-closing radius, since checking this requires only
  quantification over finite sets of words. This shows that
  left-closing is semi-decidable and the $m$ can be computed when $F$
  is left-closing. When $F$ is not left-closing, there exist $x, y$
  such that $x_{[1,\infty)} = y_{[1,\infty)}$, $x_0 \neq y_0$ and
  $F(x) = F(y)$. A standard pigeonhole argument shows that there then
  also exist such a pair of points whose left and right tails are
  eventually periodic, showing that not being left-closing is
  semidecidable.
\end{proof}

\begin{lemma}
  \label{lem:lambdaf}
  Given a left-closing cellular automaton $f : \ST^\Z \to \ST^\Z$, one
  can effectively compute the rational number $\lambda_f$ defined in
  Equation~\eqref{eq:lambda} on page \pageref{eq:lambda}.
\end{lemma}

\begin{proof}
  As observed after defining \eqref{eq:lambda}, the limit is reached
  in finite time, once $m$ is a strong left-closing radius. By the
  previous lemma, one can effectively compute a strong left-closing
  radius.
\end{proof}

\begin{theorem}
  \label{thm:decidable-slider-sweeper}
  Given a cellular automaton $f : \ST^\Z \to \ST^\Z$, it is decidable
  whether $f$ is a slider (resp. sweeper).
\end{theorem}

\begin{proof}
  By Theorem~\ref{thm:SweeperCharacterization}, a block rule is a
  sweeping rule for $f$ if and only if it is a slider rule for $f$, so
  in particular $f$ admits a slider if and only if it admits a
  sweeper. Theorem~\ref{thm:SliderCharacterization2} characterizes
  cellular automata admitting a slider as ones that are left-closing
  and satisfy $v_p(\lambda_f) \leq 0$ for all primes $p$. Decidability
  follows from the previous two lemmas.
\end{proof}

We now move on to showing that given a block rule, we can check
whether its slider or sweeper relation defines a CA.

In the rest of this section, we explain the automata-theoretic nature
of both types of rules, which allows one to decide many properties of
the slider and sweeper relations even when they do not define cellular
automata. As is a common convention in automata theory, all claims in
the rest of this section have constructive proofs (and thus imply
decidability results), unless otherwise specified.

We recall definitions from \cite{infauto} for automata on bi-infinite
words. A \emph{finite-state automaton} is $A = (Q, \ST, E, I, F)$
where $Q$ is a finite set of \emph{states}, $\ST$ the \emph{alphabet},
$E \subset Q \times \ST \times Q$ the \emph{transition relation},
$I \subset Q$ the set of \emph{initial states} and $F \subset Q$ the
set of \emph{final states}.

The pair $(Q, E)$ can be naturally seen as a labeled graph with labels
in $\ST$. The \emph{language} of such an automaton $A$ the set
$\mathcal{L}(A) \subset \ST^\Z$ of labels of bi-infinite paths in
$(Q, E)$ such that some state in $I$ is visited infinitely many times
to the left (negative indices) and some state in $F$ infinitely many
times to the right. Languages of finite-state automata are called
\emph{recognizable}.

%============ BEGIN: is this at the right position ?????
If $A \subset \ST^{-\N}$ and $B \subset \ST^{\N}$, write
$[A, B] \subset \ST^\Z$ for the set of configurations $x \in \ST^\Z$
such that for some $y \in A, z \in B$, $x_i = A_{i+1}$ for $i < 0$ and
$x_i = B_i$ for $i \geq 0$. We need the following lemma.

\begin{lemma}[Part of Proposition~IX.2.3 in \cite{infauto}]
  \label{lem:ZetaAutomatic}
  For a set $X \subset \ST^\Z$ the following are equivalent
  \begin{itemize}
  \item $X$ is recognizable
    % \item $X$ is $\zeta$-automatic
  \item $X$ is shift-invariant and a finite union of sets of the form
    $[A, B]$ where $B$ is $\omega$-recognizable (accepted by a B\"uchi
    automaton) and $A$ is the reverse of an $\omega$-recognizable set.
  \end{itemize}
\end{lemma}

%============ END: is this at the right position ?????

In the theorems of this section, note that the set
$\ST^\Z \times \ST^\Z$ is in a natural bijection with $(\ST^2)^\Z$.

\begin{proposition}
  \label{prop:arule-relation-decidable}
  Let $\chi : \ST^m \to \ST^m$ be a bijective block rule. Then the
  corresponding \emph{\arule\ relation} $F \subset (\ST^2)^\Z$ is
  recognizable.
\end{proposition}

\begin{proof}
  Let $\xi = \chi^{-1}$. The slider relation is defined as the pairs
  $y, z \in \ST^\Z$ such that for some representation $(x, 0)$ we have
  $\chi^{0-}(x) = y$ and $\xi^{0+}(x) = z$.

  For each $uw \in S^{[-m,m-1]}$ where $|u| = |w| = m$, we define
  recognizable languages
  $A_{uw} \subset (\ST^2)^{(-\infty,0)}, B_{uw} \in (\ST^2)^{[0,
    \infty)}$ such that the slider relation is
  $\bigcup_{uw \in S^{[-m,m-1]}} [A_{uw}, B_{uw}]$.

  For finite words, one-way infinite words and more generally patterns
  over any domain $D \subset \Z$, define the ordered applications of
  $\chi$ and $\xi$ (e.g. $\chi^{i+}$) with the same formulas as for
  $x \in \ST^\Z$, when they make sense.

  For each word $uw \in S^{[-m,m-1]}$, define the
  $\omega$-recognizable set $B_{uw} \subset (S^2)^\N$ containing those
  $(y, z)$ for which $\chi^{0+}(x) = z$ where $x \in \ST^\N$ satisfies
  $x_{[0,m-1]} = w$, $x_{[m,\infty)} = y_{[m,\infty)}$, and
  $\xi^{[-m+1,-1]^R}(uw)|_{[0,m-1]} = y_{[0,m-1]}$. One can easily
  construct a B\"uchi automaton recognizing this language, so $B_w$ is
  $\omega$-recognizable.

  Let then for $w \in S^m$ the set $A_w \subset (\ST^2)^{-\N}$ be
  defined as those pairs $(y, z)$ such that
  $\xi^{0-}(zw)|_{(-\infty, 0)} = y$, where
  $zw \in \ST^{(-\infty, m-1]}$. Again it is easy to construct a
  B\"uchi automaton for the reverse of $A_w$.

  Now it is straightforward to verify that the slider relation of
  $\chi$ is
  \[
    \bigcup_{uw \in S^{[-m, m-1]}} [A_{uw}, B_{uw}],
  \]
  which is recognizable by Lemma~\ref{lem:ZetaAutomatic} since the
  slider relation is always shift-invariant.
\end{proof}

\begin{lemma}
  \label{lem:FunctionDecidable}
  Given a recognizable set $X \subset (S^2)^\Z$, interpreted as a
  binary relation over $S^\Z$, it is decidable whether $X$ defines a
  function.
\end{lemma}

\begin{proof}
  Since recognizable sets representing relations are closed under
  Cartesian products, projections and intersections (by standard
  constructions), if $X$ is recognizable also the `fiber product'
  $Y \subset (S^2)^\Z$ containing those pairs $(z,z')$ satisfying
  $\exists y: (y,z) \in X \wedge (y,z') \in X$ is recognizable. The
  diagonal $\Delta$ of $(S^2)^\Z$ containing all pairs of the form
  $(z,z)$ is also clearly recognizable.

  Since recognizable languages are closed under complementation
  \cite{infauto}, we obtain that $((S^2)^\Z \setminus \Delta) \cap Y$
  is recognizable. This set is empty if and only if $X$ is a function,
  proving decidability, since all proofs in this section are
  constructive and emptiness of a recognizable language is decidable
  using standard graph algorithms.
\end{proof}

The following is a direct corollary.

\begin{theorem}
  Given a block rule, it is decidable whether it is the sliding rule
  of a CA.
\end{theorem}

We now discuss sweeping rules.

\begin{proposition}
  \label{prop:brule-relation-decidable}
  Let $\chi : \ST^m \to \ST^m$ be a block rule. Then the corresponding
  \emph{\brule\ relation} $F \subset (\ST^2)^\Z$ is recognizable.
\end{proposition}

\begin{proof}
  One can easily construct a finite-state automaton accepting the
  language $X \subset (\{0,1\}^2 \times \ST^2)^\Z$ containing those
  $(x, x', y, z) \in (\{0,1\}^2 \times \ST^2)^\Z$ where
  \[
    \chi^{m+}(y)|_{[n,\infty)} = z|_{[n, \infty)}
  \]
  and $x_i = 1 \iff i = m$ and $x'_i = 1 \iff i = n$. Simply construct
  an automaton that checks that there is exactly one $1$-symbol on
  each of the first two tracks, and when it sees the first $1$ is seen
  on the first track it starts keeping in its state the current
  contents of the active window (where the block rule is being
  applied). When $1$ is seen on the second track, it also starts
  checking that the image is correct.

  Since $X$ is described by an automaton and $|\ST| \leq 2^k$ for some
  $k$, an adaptation of \cite[Theorem~IX.7.1]{infauto} shows that
  there exists a monadic second-order formula over the successor
  function of $\Z$, i.e. some formula $\phi \in \mbox{MF}_2(<)$, that
  defines those tuples sets of integers
  $(x, x', y_1, ..., y_k, z_1, ..., z_k)$ where $(y_1,...,y_k)$ codes
  some $y$ and $(z_1,...,z_k)$ some $z$ such that $(x, x', y, z)$ is
  in $X$.

  Since in tuple $(x, x', y_1, ..., y_k, z_1, ..., z_k)$ that
  satisfies $\phi$ we have $|x| = |x|' = 1$, it is standard to modify
  $\phi'$ into a formula where $x, x'$ are replaced by first-order
  variables $i, j$ and correspond to the unique places in $x$ and $x'$
  where the unique $1$ appears. Now the formula $\psi$ defined by
  \[
    \forall j \in \Z: \forall n \in \Z: \exists i \leq n: \phi'(i, j,
    y_1,...,y_k, z_1,...,z_k)
  \]
  defines those tuples $(y_1, ..., y_k, z_1, ..., z_k)$ that code
  pairs $(y, z)$ which are in the sweeper relation for $\chi$. Another
  application of \cite[Theorem~IX.7.1]{infauto} then shows that
  sweeper relation is recognizable.
\end{proof}

The sweeping relation need not be closed, as shown in
Example~\ref{ex:NotClosed}. However, whether it is closed is
decidable.

\begin{lemma}
  \label{lem:ClosedDecidable}
  Given a recognizable $X \subset \ST^\Z$, it is decidable whether $X$
  is closed.
\end{lemma}

\begin{proof}
  Take an automaton recognizing $X$, remove alls states from which an
  initial state is not reachable to the left, and ones from which a
  final state is not reachable to the right. Turn all states into
  initial and final states. Now $X$ is closed if and only if the new
  automaton recognizes $X$, which is decidable by standard arguments.
\end{proof}

\begin{theorem}
  \label{thm:block-decidable-whether-sweeping}
  Given a block rule, it is decidable whether the sweeping relation it
  defines is a CA.
\end{theorem}

\begin{proof}
  The sweeping rule of a block rule defines a CA if and only if the
  sweeping relation is closed and defines a function. These are
  decidable by Lemma~\ref{lem:FunctionDecidable} and
  Lemma~\ref{lem:ClosedDecidable}, respectively.
\end{proof}

%======================================================================
\section{Future work and open problems}

To obtain a practical computer implementation method for cellular
automata, one would need much more work. The radius of $\chi$ should
be given precise bounds, and we would also need bounds on how long it
takes until the sweep starts producing correct values. Future work
will involve clarifying the connection between the radii $m$ of local
rules $\chi : S^m \to S^m$ and the strong left-closing radii, the
study of non-bijective local rules, and the study of sweeping rules on
periodic configurations.

On the side of theory, it was shown in Section~\ref{sec:rev-ca} that
the hierarchy of left- and right-closing cellular automata corresponds
to the hierarchy of sweeps starting from the second level. Neither
hierarchy collapses on the first level, since there exists CA which
are left-closing but not right-closing, from which one also obtains CA
which are in $\mathcal{L}_1$ but not $\mathcal{R}_1$.

\begin{question}
  Does the hierarchy collapse on a finite level? Is every surjective
  CA in this hierarchy?
\end{question}

As we do not know which cellular automata appear on which levels, we
do not know whether these levels are decidable. For example we do not
know whether it is decidable if a given CA is the composition of a
left sweep and a right sweep.

It seems likely that the theory of sliders can be extended to shifts
of finite type. If $X$ is a subshift, say that a homeomorphism
$\chi : X \to X$ is \emph{local} if its application modifies only a
(uniformly) bounded set of coordinates. One can define sliding
applications of such homeomorphisms exactly as in the case of
$\ST^\Z$.

\begin{question}
  Let $X \subset \ST^\Z$ be a transitive subshift of finite type.
  Which endomorphisms of $X$ are defined by a sliding rule defined by
  a local homeomorphism?
\end{question}

In \cite{blockCA}, block representations are obtained for cellular
automata in one and two dimensions, by considering the set of stairs
of reversible cellular automata. Since stairs play a fundamental role
for sliders as well, it seems natural to attempt to generalize our
theory to higher dimensions.

%======================================================================
\par

\noindent
\textbf{Acknowledgement.} The authors gratefully acknowledge partial
support for this work by two short term scientific missions of the EU
COST Action IC1405.
%======================================================================

\bibliographystyle{plain}
\bibliography{bib}

\end{document}